\documentclass[11pt,reqno]{amsart}

\usepackage{amsaddr}

\usepackage[margin=1.5in]{geometry}
\usepackage[colorinlistoftodos]{todonotes}

\newcounter{todocounter}

\usepackage[T1]{fontenc}
\usepackage{tgpagella}
\usepackage{mathpazo}

\usepackage[pagebackref]{hyperref} 
\usepackage{xcolor}
\usepackage{amsmath,amsthm}
\usepackage{amssymb}
\usepackage{graphicx}
\usepackage[mathscr]{eucal}
\usepackage{MnSymbol}
\usepackage[frame,ps,matrix,arrow,curve,rotate,all,2cell,tips]{xy}
\usepackage{epic,eepic}
\usepackage{enumerate}

\newtheorem{main}{Theorem}

\newtheorem{theorem}[equation]{Theorem}
\newtheorem{lemma}[equation]{Lemma}

\newtheorem{proposition}[equation]{Proposition}
\newtheorem{corollary}[equation]{Corollary}

\theoremstyle{definition}
\newtheorem{definition}[equation]{Definition}
\newtheorem{example}[equation]{Example}
\newtheorem{remark}[equation]{Remark}

\newtheorem{convention}[equation]{Convention}

\numberwithin{equation}{subsection}

\newcommand{\Q}{\mathbb{Q}}

\newcommand{\M}{\mathcal{M}}
\newcommand{\Mcof}{\M_{\cof}}
\newcommand{\mtoc}{\M^{\fC}}
\newcommand{\Msigman}{\M^{\Sigma_n}}
\newcommand{\Msigmanop}{\M^{\sigmaop_n}}

\newcommand{\Msigmabopd}{\M^{\sigmaop_{[\ub]} \times \{d\}}}

\newcommand{\N}{\mathcal{N}}
\newcommand{\ntoc}{\N^{\fC}}
\newcommand{\Nsigmabopd}{\N^{\sigmaop_{[\ub]} \times \{d\}}}

\newcommand{\C}{\mathcal{C}}

\newcommand{\D}{\mathcal{D}}

\newcommand{\sP}{\mathscr{P}}

\newcommand{\sE}{\mathscr{E}}

\renewcommand{\emptyset}{\varnothing}

\renewcommand{\tilde}[1]{\widetilde{#1}}

\DeclareMathOperator{\colim}{colim}

\newcommand{\po}{\ar@{}[dr]|(.7){\Searrow}}
\newcommand{\pb}{\ar@{}[dr]|(.3){\Nwarrow}}

\newcommand{\ch}{\mathsf{ch}}
\newcommand{\Ch}{\mathsf{Ch}}
\newcommand{\Vect}{\mathsf{Vect}}
\newcommand{\Ab}{\mathsf{Ab}}
\newcommand{\Fin}{\mathsf{Fin}}
\newcommand{\Chk}{\Ch(k)}
\newcommand{\Chkplus}{\Chk_{\geq 0}}
\DeclareMathOperator{\ev}{Ev}

\newcommand{\boxprod}{\mathbin\square}

\newcommand{\sAb}{\mathsf{sAb}}
\newcommand{\Sp}{\mathsf{Sp}}

\newcommand{\coprodover}[1]{\underset{#1}{\coprod}}

\newcommand{\tensorover}[1]{\underset{#1}{\otimes}}
\newcommand{\tensoroversigman}{\underset{\Sigma_n}{\otimes}}
\newcommand{\tensoroversigmar}{\underset{\Sigma_r}{\otimes}}

\newcommand{\algo}{\alg(\sO)}
\newcommand{\algp}{\alg(\sP)}

\newcommand{\lcm}{L_\C(\M)}
\newcommand{\ldn}{L_\D(\N)}

\newcommand{\algolcm}{\alg(\sO;\lcm)}
\newcommand{\algom}{\alg(\sO;\M)}
\newcommand{\algpldn}{\alg(\sP;\ldn)}
\newcommand{\algpn}{\alg(\sP;\N)}
\newcommand{\algplcm}{\alg(\sP;\lcm)}

\usepackage{tikz}
\usetikzlibrary{arrows,decorations.pathmorphing}
\usetikzlibrary{backgrounds,positioning}
\usetikzlibrary{fit,petri,shapes.misc}

\tikzset{auto}

\tikzset{empty/.style={circle,inner sep=0pt,minimum size=6mm}}
\tikzset{emptyvt/.style={circle,inner sep=0pt,minimum size=0mm}}

\tikzset{plain/.style={circle,draw,very thick,
inner sep=0pt,minimum size=6mm}}

\tikzset{fatplain/.style={rounded rectangle,draw,very thick,minimum size=6mm}}

\tikzset{bigplain/.style={rounded rectangle,draw,very thick,minimum size=.8cm}}

\tikzset{yellowvt/.style={circle,draw,fill=yellow,very thick,inner sep=0pt,minimum size=6mm}}

\tikzset{bluevt/.style={circle,draw,fill=blue!20,very thick,inner sep=0pt,minimum size=6mm}}

\tikzset{greenvt/.style={circle,draw,fill=green!30,very thick,inner sep=0pt,minimum size=6mm}}

\tikzset{redvt/.style={circle,draw,fill=red!30,very thick,inner sep=0pt,minimum size=6mm}}

\tikzset{arrow/.style={->,thick}}
\tikzset{dashedarrow/.style={->,dashed,thick}}
\tikzset{dottedarrow/.style={->,dotted,thick}}
\tikzset{mapto/.style={|->,thick}}

\tikzset{implies/.style={thick,double,double equal sign distance,-implies}}

\tikzset{line/.style={thick}}
\tikzset{dottedline/.style={dotted,thick}}
\tikzset{dashedline/.style={dashed,thick}}

\tikzset{inputleg/.style={<-,thick}}
\tikzset{outputleg/.style={->,thick}}
\tikzset{dottedinput/.style={<-,dotted,thick}}

\newcommand{\adjoint}{\hspace{-.1cm}
\nicearrow\xymatrix{ \ar@<2pt>[r] & \ar@<2pt>[l]}\hspace{-.1cm}}
\renewcommand{\hookrightarrow}{\nicexy{\ar@{^{(}->}[r] &}}
\newcommand{\nicearrow}{\SelectTips{cm}{10}}
\newcommand{\nicexy}{\nicearrow\xymatrix@C+10pt@R+10pt}
\newcommand{\narrowxy}{\nicearrow\xymatrix@R+10pt}

\newcommand{\pushout}{\ar@{}[dr]|(0.75){\Searrow}}
\newcommand{\drrpushout}{\ar@{}[drr]|(0.90){\Searrow}}

\newcommand{\comp}{\circ}
\newcommand{\defn}{\,\overset{\mathrm{def}}{=\joinrel=}\,}
\newcommand{\fbar}{\overline{f}}

\newcommand{\Rbar}{\overline{R}}
\renewcommand{\to}{\hspace{-.1cm}\nicearrow\xymatrix@C-.2cm{\ar[r]&}\hspace{-.1cm}}

\newcommand{\tensorunit}{\mathbb{1}}
\newcommand{\tensorunitm}{\tensorunit^{\M}}
\newcommand{\tensorunitn}{\tensorunit^{\N}}

\newcommand{\fB}{\mathfrak{B}}
\newcommand{\frakC}{\mathfrak{C}}
\newcommand{\fC}{\mathfrak{C}}

\newcommand{\fG}{\mathfrak{G}}

\newcommand{\sA}{\mathsf{A}}

\renewcommand{\sE}{\mathsf{E}}

\newcommand{\sL}{\mathsf{L}}

\newcommand{\Monoid}{\mathsf{Monoid}}
\newcommand{\CMonoid}{\mathsf{CMonoid}}

\newcommand{\sO}{\mathsf{O}}
\newcommand{\osuba}{\mathsf{O}_{A}}
\newcommand{\osubazero}{\mathsf{O}_{A}^0}
\newcommand{\osubaone}{\mathsf{O}_{A}^1}
\newcommand{\osubatwo}{\mathsf{O}_{A}^2}
\newcommand{\osubatminusone}{\mathsf{O}_{A}^{t-1}}
\newcommand{\osubat}{\mathsf{O}_{A}^{t}}
\newcommand{\osubainfinity}{\mathsf{O}_{A_{\infty}}}

\renewcommand{\sP}{\mathsf{P}}

\newcommand{\initialo}{\varnothing_{\sO}}
\newcommand{\initialp}{\varnothing_{\sP}}

\newcommand{\As}{\mathsf{As}}
\newcommand{\Com}{\mathsf{Com}}
\newcommand{\Lie}{\mathsf{Lie}}

\newcommand{\scrI}{\mathscr{I}}

\newcommand{\Lbar}{\overline{L}}

\newcommand{\varphibar}{\overline{\varphi}}

\newcommand{\ua}{\underline{a}}
\newcommand{\ub}{\underline{b}}
\newcommand{\uc}{\underline{c}}

\newcommand{\smallop}{{\scalebox{.5}{$\mathrm{op}$}}}
\newcommand{\cof}{{\scalebox{.5}{$\mathrm{cof}$}}}
\newcommand{\acof}{{\scalebox{.5}{$\mathrm{t.cof}$}}}
\newcommand{\tcof}{{\scalebox{.5}{$\mathrm{t.cof}$}}}
\newcommand{\clubcof}{(\clubsuit)_{\cof}}
\newcommand{\clubacof}{(\clubsuit)_{\acof}}
\newcommand{\clubtcof}{(\clubsuit)_{\tcof}}

\newcommand{\cald}{\mathcal{D}}
\newcommand{\caldop}{\mathcal{D}^{\smallop}}

\newcommand{\calm}{\mathcal{M}}
\newcommand{\calmc}{\calm^{\fC}}

\newcommand{\caln}{\mathcal{N}}

\newcommand{\set}{\mathsf{Set}}

\newcommand{\symseq}{\mathsf{SymSeq}}
\newcommand{\symseqc}{\symseq_{\fC}}
\newcommand{\symseqcm}{\symseqc(\calm)}
\newcommand{\symseqcn}{\symseqc(\N)}

\newcommand{\alg}{\mathsf{Alg}}
\newcommand{\Alg}{\mathsf{Alg}}
\newcommand{\calg}{\mathsf{CAlg}}

\newcommand{\sigmaop}{\Sigma^{\smallop}}

\newcommand{\sigmaopb}{\Sigma^{\smallop}_{\smallbrb}}
\newcommand{\sigmaopc}{\Sigma^{\smallop}_{\smallbrc}}

\newcommand{\catc}{\mathsf{Cat}^{\fC}}

\newcommand{\operad}{\mathsf{Operad}}
\newcommand{\operadc}{\operad^{\fC}}
\newcommand{\woperadc}{\operad^{\fC \rcirclearrowdown}}

\newcommand{\omegaoperad}{\operad^{\Omega}}

\newcommand{\cyoperad}{\operad^{\fC}_{\mathsf{cyc}}}
\newcommand{\modoperad}{\operad^{\fC}_{\mathsf{mod}}}

\newcommand{\halfpropc}{\frac{1}{2}\mathsf{Prop}^{\fC}}

\newcommand{\dioperad}{\mathsf{Dioperad}}
\newcommand{\dioperadc}{\dioperad^{\fC}}

\newcommand{\properadc}{\mathsf{Properad}^{\fC}}
\newcommand{\wproperadc}{\mathsf{Properad}^{\fC \rcirclearrowdown}}

\newcommand{\propc}{\mathsf{Prop}^{\fC}}
\newcommand{\wpropc}{\mathsf{Prop}^{\fC \rcirclearrowdown}}

\newcommand{\gprop}{\mathsf{Prop}^{\fG}}

\newcommand{\pofc}{\Sigma_{\frakC}}
\newcommand{\pofcop}
{\pofc^{\scalebox{.6}{$\mathrm{op}$}}}

\newcommand{\prof}{\mathsf{Prof}}
\newcommand{\profc}{\prof(\fC)}

\renewcommand{\pb}{\mathcal{P}(\fB)}

\newcommand{\smallprof}[1]
{\raisebox{.05cm}{\scalebox{0.8}{#1}}}

\newcommand{\cjbrbj}
{\smallprof{$\binom{c_j}{[\ub_j]}$}}

\newcommand{\singledbrtbbrc}
{\smallprof{$\binom{d}{[tb]; [\uc]}$}}

\newcommand{\singledaprimecprime}
{\smallprof{$\binom{d}{\ua',\uc'}$}}

\newcommand{\singledbrb}
{\smallprof{$\binom{d}{[\ub]}$}}

\newcommand{\drcub}
{\smallprof{$\binom{d}{[rc];[\ub]}$}}
\newcommand{\duarcub}
{\smallprof{$\binom{d}{[\ua];[rc];[\ub]}$}}
\newcommand{\dqcub}
{\smallprof{$\binom{d}{[qc];[\ub]}$}}

\newcommand{\duc}
{\smallprof{$\binom{d}{\uc}$}}

\newcommand{\singledbrc}
{\smallprof{$\binom{d}{[\uc]}$}}

\newcommand{\singledbrtc}
{\smallprof{$\binom{d}{[tc]}$}}
\newcommand{\singledbrabrc}
{\smallprof{$\binom{d}{[\ua]; [\uc]}$}}

\newcommand{\singledempty}
{\smallprof{$\binom{d}{\varnothing}$}}
\newcommand{\dnothing}
{\smallprof{$\binom{d}{\varnothing}$}}

\newcommand{\smallbr}[1]
{\raisebox{.03cm}{\scalebox{0.5}{#1}}}

\newcommand{\brb}{[\ub]}

\newcommand{\smallbra}{\smallbr{$[\ua]$}}

\newcommand{\smallbrb}{\smallbr{$[\ub]$}}

\newcommand{\smallbrc}{\smallbr{$[\uc]$}}

\newcommand{\dbrb}{\smallprof{$\binom{d}{\brb}$}}

\newcommand{\sigmabra}{\Sigma_{\smallbr{$[\ua]$}}}

\newcommand{\sigmabrb}{\Sigma_{\smallbr{$[\ub]$}}}

\newcommand{\sigmabrbj}{\Sigma_{\smallbr{$[\ub_j]$}}}

\newcommand{\sigmabrc}{\Sigma_{\smallbr{$[\uc]$}}}

\newcommand{\sigmabraop}{\sigmabra^{\smallop}}

\newcommand{\sigmabrbop}{\sigmabrb^{\smallop}}
\newcommand{\sigmabrbjop}{\sigmabrbj^{\smallop}}

\newcommand{\sigmabrcop}{\sigmabrc^{\smallop}}

\newcommand{\sigmabrcopd}{\sigmabrcop \times \{d\}}

\newcommand{\sigmaofc}{\pofc}
\newcommand{\sigmacop}{\pofcop}
\newcommand{\sigmacopc}{\sigmacop \times \fC}

\newcommand{\dbrch}{([\uc];d)}

\newcommand{\andspace}{\qquad\text{and}\qquad}

\renewcommand{\lim}{\mathsf{lim}\,}

\DeclareMathOperator{\Hom}{Hom}

\DeclareMathOperator{\id}{id}
\DeclareMathOperator{\Id}{Id}
\DeclareMathOperator{\Kan}{\mathsf{Kan}}
\DeclareMathOperator{\Mod}{\mathsf{Mod}}

\begin{document}

\title{Homotopical Adjoint Lifting Theorem}

\author{David White}
\address{Denison University
\\ Granville, OH}
\email{david.white@denison.edu}

\author{Donald Yau}
\address{The Ohio State University at Newark \\ Newark, OH}
\email{yau.22@osu.edu}

\begin{abstract}
This paper provides a homotopical version of the adjoint lifting theorem in category theory, allowing for Quillen equivalences to be lifted from monoidal model categories to categories of algebras over colored operads. The generality of our approach allows us to simultaneously answer questions of rectification and of changing the base model category to a Quillen equivalent one. We work in the setting of colored operads, and we do not require them to be $\Sigma$-cofibrant. Special cases of our main theorem recover many known results regarding rectification and change of model category, as well as numerous new results.  In particular, we recover a recent result of Richter-Shipley about a zig-zag of Quillen equivalences between commutative $H\Q$-algebra spectra and commutative differential graded $\Q$-algebras, but our version involves only three Quillen equivalences instead of six. We also work out the theory of how to lift Quillen equivalences to categories of colored operad algebras after a left Bousfield localization.
\end{abstract}


\maketitle

\tableofcontents

\section{Introduction}

When studying the homotopy theory of algebras over operads, a common question is that of rectification, i.e., determining when a weak equivalence $f:\sO \to \sP$ of operads induces a Quillen equivalence between $\sO$-algebras and $\sP$-algebras. Rectification can be viewed as a generalization of change-of-rings results. A question that has received less attention, but which is an important part of the theory, regards changing the base model category. When does a Quillen equivalence $L:\M \adjoint \N:R$ lift to an equivalence on the level of algebras? This question was first studied in \cite{ss03}, has been studied in a limited scope for $\Sigma$-cofibrant operads in \cite{fresse-book}, and has been studied for commutative monoids in \cite{white-commutative}, but a general treatment is lacking in the literature.

In this paper, we provide a unified framework for answering questions of rectification and change of base model category simultaneously and in a great deal of generality. We work with $\mathfrak{C}$-colored operads for any set $\fC$, and we use model categories and semi-model categories as our setting for studying the homotopy theory of operad algebras. Relevant definitions are reviewed in Sections \ref{sec:prelims} and \ref{sec:colored}. Fundamentally, we are interested in studying the adjoint lifting diagram
\begin{equation*}
\nicexy{\algo \ar@<2.5pt>[r]^-{\Lbar} \ar@<2.5pt>[d]^-{U}
& \algp \ar@<2.5pt>[l]^-{R} \ar@<2.5pt>[d]^-{U} \\
\mtoc \ar@<2.5pt>[r]^-{L} \ar@<2.5pt>[u]^-{\sO \comp -}  
& \ntoc \ar@<2.5pt>[l]^-{R} \ar@<2.5pt>[u]^-{\sP \comp -}}
\end{equation*}
from a homotopical perspective, in the context of Quillen equivalences and operadic algebras. We determine when a Quillen equivalence $(L,R)$ induces a Quillen equivalence $(\overline{L}, R)$ of algebras. In Section \ref{sec:main}, we prove our Main Theorem \ref{main.theorem}, stated here:

\begin{main}[Lifting Quillen Equivalences]
Suppose:
\begin{enumerate}
\item $L : \M \adjoint \N : R$ is a nice Quillen equivalence (Def. \ref{def:nice.qeq}).
\item $f : \sO \to R\sP$ is a map of $\fC$-colored operads in $\M$ with $\fC$ a set, $\sO$ an entrywise cofibrant $\fC$-colored operad in $\M$, and $\sP$ an entrywise cofibrant $\fC$-colored operad in $\N$.  The entrywise adjoint $\fbar : L\sO \to \sP$ is an entrywise weak equivalence in $\N$. 
\end{enumerate}
Then the lifted adjunction \eqref{lbar.ocomp.diagram}
\[\nicexy{\algo \ar@<2.5pt>[r]^-{\Lbar} & \algp \ar@<2.5pt>[l]^-{R}}\]
is a Quillen equivalence between the semi-model categories of $\sO$-algebras in $\M$ and of $\sP$-algebras in $\N$ (Theorem \ref{theorem623}).
\end{main}

We provide numerous examples of model categories satisfying the conditions of this theorem, and we prove a version of this theorem for $\Sigma$-cofibrant colored operads (Theorem \ref{main.theorem.Sigma}), where the hypotheses on the adjunction $(L,R)$ are effectively always satisfied in examples of interest.

In Section \ref{sec:rect-and-change}, we specialize this general theorem and the $\Sigma$-cofibrant version to obtain results about rectification, and about lifting Quillen equivalences to modules, (commutative) algebras,  non-symmetric operads, generalized props, cyclic operads, and modular operads. We recover results from \cite{ss03}, \cite{muro11}, \cite{muro14}, and \cite{fresse-book} all as special cases of the same general theorem, and we then provide numerous new applications of this theorem. The results for non $\Sigma$-cofibrant operads are entirely new. The results of this paper have been applied \cite{GW, white-yau4, yau}.

In Section \ref{sec:bous-loc}, we apply our main theorem in left Bousfield localizations $\lcm$ and $\ldn$, so that we may lift Quillen equivalences $L:\M \adjoint \N:R$ to local categories of algebras
\[\overline{L}: \algolcm \adjoint \algpldn:R.\]
In Theorems \ref{main.theorem.local} and \ref{main.theorem.local.Sigma}, we provide checkable conditions allowing for the local application of our main results,Theorems \ref{main.theorem} and \ref{main.theorem.Sigma}.  We provide examples of model categories where the conditions are satisfied. We specialize these local results to obtain results about local rectification, local change-of-rings, and local modules, (commutative) algebras,  non-symmetric operads,  generalized props, cyclic operads, and modular operads. 

Lastly, in Section \ref{sec:applications}, we recover results of \cite{ss03}, \cite{shipley-hz-spectra}, and \cite{richter-shipley}, where chains of Quillen equivalences were manually lifted to categories of modules, (commutative) monoids, and $E_\infty$-algebras. These examples include the Dold-Kan equivalence, the Quillen equivalence between DGAs and $HR$-algebra spectra, and the Quillen equivalence between commutative DGAs and commutative $HR$-algebra spectra. We demonstrate how our main theorems could have been used in these settings.  In particular, as special cases of our main theorem in the characteristic $0$ setting, we obtain a zig-zag of three Quillen equivalences between commutative $H\Q$-algebra spectra and commutative differential graded $\Q$-algebras.  This is a slightly improved version of \cite{richter-shipley} (Corollary 8.4), which contains a zig-zag of six Quillen equivalences between the same end categories.   We hope our unified approach will find many applications in future such settings.


\section{Model Categories}
\label{sec:prelims}

In this section we recall some key concepts in model category theory.  Our main references here are \cite{hirschhorn,hovey,ss,ss03}.  In this paper, $(\M, \otimes, \tensorunit, \Hom)$ and $\N$ will usually be a symmetric monoidal closed category with $\otimes$-unit $\tensorunit$ and internal hom $\Hom$.  We assume $\M$ and $\N$ have all small limits and colimits.  Their initial and terminal objects are denoted by $\varnothing$ and $*$, respectively.

\subsection{Monoidal Model Categories}

We assume the reader is familiar with basic facts about model categories as presented in \cite{hirschhorn} and \cite{hovey}.  For a model category $\M$, its subcategory of cofibrations is denoted by $\Mcof$. When we work with model categories they will most often be \emph{cofibrantly generated}; i.e., there is a set $I$ of cofibrations and a set $J$ of trivial cofibrations (i.e. maps which are both cofibrations and weak equivalences) which permit the small object argument (with respect to some cardinal $\kappa$), and a map is a (trivial) fibration if and only if it satisfies the right lifting property with respect to all maps in $J$ (resp. $I$).

Let $I$-cell denote the class of transfinite compositions of pushouts of maps in $I$, and let $I$-cof denote retracts of such. In order to run the small object argument, we will assume the domains $K$ of the maps in $I$ (and $J$) are $\kappa$-small relative to $I$-cell (resp. $J$-cell); i.e., given a regular cardinal $\lambda \geq \kappa$ and any $\lambda$-sequence $X_0\to X_1\to \cdots$ formed of maps $X_\beta \to X_{\beta+1}$ in $I$-cell, the map of sets
\[
\nicexy{\colim_{\beta < \lambda} \M\bigl(K,X_\beta\bigr) \ar[r] 
& \M\bigl(K,\colim_{\beta < \lambda} X_\beta\bigr)}
\]
is a bijection. An object is \emph{small} if there is some $\kappa$ for which it is $\kappa$-small. See Chapter 10 of \cite{hirschhorn} for a more thorough treatment of this material.

We must now discuss the interplay between the monoidal structure and the model structure which we will require in this paper. This definition is taken from 3.1 in \cite{ss}.

\begin{definition}
A symmetric monoidal closed category $\calm$ equipped with a model structure is called a \emph{monoidal model category} if it satisfies the following \emph{pushout product axiom}: 

\begin{itemize}
\item Given any cofibrations $f:X_0\to X_1$ and $g:Y_0\to Y_1$, the pushout corner map
\[
\nicexy{
X_0\otimes Y_1 \coprod\limits_{X_0\otimes Y_0}X_1\otimes Y_0 
\ar[r]^-{f\boxprod g} & X_1\otimes Y_1}
\]
is a cofibration. If, in addition, either $f$ or $g$ is a weak equivalence then $f\boxprod g$ is a trivial cofibration.
\end{itemize}
\end{definition}

Note that the pushout product axiom is equivalent to the statement that $-\otimes-$ is a Quillen bifunctor.

\subsection{Semi-Model Categories}

When attempting to study the homotopy theory of algebras over a colored operad, the usual method is to transfer a model structure from $\calm$ to this category of algebras along the free-forgetful adjunction (using Kan's Lifting Theorem \cite{hirschhorn} (11.3.2)). Unfortunately, it is often the case that one of the conditions for Kan's theorem cannot be checked fully, so that the resulting homotopical structure on the category of algebras is something less than a model category. This type of structure was first studied in \cite{hovey-monoidal} and \cite{spitzweck-thesis}, and later in published sources such as \cite{fresse} and \cite{fresse-book} (12.1).

\begin{definition}
\label{defn:semi-model-cat}
Assume there is an adjunction $F:\calm \adjoint \D:U$ where $\calm$ is a cofibrantly generated model category, $\D$ is bicomplete, and $U$ preserves colimits indexed by non-empty ordinals. 

We say that $\D$ is a \emph{semi-model category} if $\D$ has three classes of morphisms called \emph{weak equivalences}, \emph{fibrations}, and \emph{cofibrations} such that the following axioms are satisfied.   A \emph{cofibrant} object $X$ means an object in $\D$ such that the map from the initial object of $\D$ to $X$ is a cofibration in $\D$.  Likewise, a \emph{fibrant} object is an object for which the map to the terminal object in $\D$ is a fibration in $\D$.

\begin{enumerate}
\item $U$ preserves fibrations and trivial fibrations ($=$ maps that are both weak equivalences and fibrations).
\item $\D$ satisfies the 2-out-of-3 axiom and the retract axiom, i.e. if a morphism $f$ is a retract of $g$, and $g$ is a weak equivalence, fibration, or cofibration, then so is $f$.
\item Cofibrations in $\D$ have the left lifting property with respect to trivial fibrations. Trivial cofibrations ($=$ maps that are both weak equivalences and cofibrations) in $\D$ whose domain is cofibrant have the left lifting property with respect to fibrations.
\item Every map in $\D$ can be functorially factored into a cofibration followed by a trivial fibration. Every map in $\D$ whose domain is cofibrant can be functorially factored into a trivial cofibration followed by a fibration.
\item The initial object in $\D$ is cofibrant.
\item Fibrations and trivial fibrations are closed under pullback.
\end{enumerate}

$\D$ is said to be \textit{cofibrantly generated} if there are sets of morphisms $I'$ and $J'$ in $\D$ such that the following conditions are satisfied.
\begin{enumerate}
\item
Denote by $I'$-inj the class of maps that have the right lifting property with respect to maps in $I'$.  Then $I'$-inj is the class of trivial fibrations.
\item
$J'$-inj is the class of fibrations in $\D$.
\item
The domains of $I'$ are small relative to $I'$-cell.
\item
The domains of $J'$ are small relative to maps in $J'$-cell whose domain is sent by $U$ to a cofibrant object in $\calm$.
\end{enumerate}
\end{definition}

In practice the weak equivalences (resp. fibrations) are morphisms $f$ such that $U(f)$ is a weak equivalence (resp. fibration) in $\calm$, and the generating (trivial) cofibrations of $\D$ are maps of the form $F(I)$ and $F(J)$ where $I$ and $J$ are the generating (trivial) cofibrations of $\calm$.

Note that the only difference between a semi-model structure and a model structure is that one of the lifting properties and one of the factorization properties requires the domain of the map in question to be cofibrant. Because fibrant and cofibrant replacements are constructed via factorization, (4) of a semi-model category implies that every object has a cofibrant replacement and that cofibrant objects have fibrant replacements. So one could construct a fibrant replacement functor which first does cofibrant replacement and then does fibrant replacement. These functors behave as they would in the presence of a full model structure.

\subsection{Quillen Adjunctions and Quillen Equivalences}

An adjunction with left adjoint $L$ and right adjoint $R$ is denoted by $L \dashv R$.

\begin{definition}
A \emph{lax monoidal functor} $F : \M \to \N$ between two monoidal categories is a functor equipped with structure maps
\[\nicexy{FX \otimes FY \ar[r]^-{F^2_{X,Y}} & F(X \otimes Y), \quad \tensorunit^{\N} \ar[r]^-{F^0} & F\tensorunit^{\M}}\]
for $X$ and $Y$ in $\M$ that are associative and unital in a suitable sense \cite{maclane} (XI.2).  If, furthermore, $\M$ and $\N$ are symmetric monoidal categories, and $F^2$ is compatible with the symmetry isomorphisms, then $F$ is called a \emph{lax symmetric monoidal functor}.  If the structure maps $F^2$ and $F^0$ are isomorphisms, then $F$ is called a \emph{strong monoidal functor}.
\end{definition}

Note that what is called a lax monoidal functor here is simply called a monoidal functor in \cite{maclane}.

\begin{definition}
Suppose  $L : \M \adjoint \N : R$ is an adjunction between monoidal categories with $R$ a lax monoidal functor.  For objects $X$ and $Y$ in $\M$, the map
\begin{equation}\label{comonoidal.map}
\nicexy{L\left(X \otimes Y\right) \ar[r]^-{L^2_{X,Y}} & LX \otimes LY \in \N,}
\end{equation}
defined as the adjoint of the composite
\[\nicexy{X \otimes Y \ar[r]^-{(\eta_X, \eta_Y)} & RLX \otimes RLY \ar[r]^-{R^2_{X,Y}} & R\bigl(LX \otimes LY\bigr),}\]
is called the \emph{comonoidal structure map} of $L$ \cite{ss03} (3.3).  Here $\eta$ is the unit of the adjunction.
\end{definition}

The following definition applies to both model categories and semi-model categories.

\begin{definition}\label{quillen.pair}
Suppose  $L : \M \adjoint \N : R$ is an adjunction between (semi) model categories.
\begin{enumerate}
\item We call $L \dashv R$ a \emph{Quillen adjunction} if the right adjoint $R$ preserves fibrations and trivial fibrations.  In this case, we call $L$ a \emph{left Quillen functor} and $R$ a \emph{right Quillen functor}.
\item We call a Quillen adjunction $L \dashv R$ a \emph{Quillen equivalence} if, for each map $f : LX \to Y \in \N$ with $X$ cofibrant in $\M$ and $Y$ fibrant in $\N$, $f$ is a weak equivalence in $\N$ if and only if its adjoint $\fbar : X \to RY$ is a weak equivalence in $\M$.  In this case, we call $L$ a \emph{left Quillen equivalence} and $R$ a \emph{right Quillen equivalence}.
\end{enumerate}
\end{definition}

\begin{remark}
For model categories, our definition of a Quillen equivalence is of course the standard one \cite{hovey} (1.3.12).  On the other hand, for semi-model categories, our definition of a Quillen equivalence is actually slightly stronger than (and hence implies) the one given in \cite{fresse-book} (12.1.8, p.191).  To see that our definition of a Quillen equivalence between semi-model categories is stronger than Fresse's, simply use the usual proof  \cite{hovey} (1.3.13 (a) $\Rightarrow$ (b)) that there are several equivalent ways to state a Quillen equivalence for model categories.  As a consequence, the total derived functors of a Quillen equivalence between semi-model categories (as in Def. \ref{quillen.pair} (2)) form adjoint equivalences between the homotopy categories.
\end{remark}

\begin{definition}
\label{def:weak.symmetric.monoidal}
Suppose  $L : \M \adjoint \N : R$ is a Quillen adjunction between monoidal  categories that are also model categories in which $R$ is equipped with a lax (symmetric) monoidal structure.  We call the Quillen adjunction $L \dashv R$ a \emph{weak (symmetric) monoidal Quillen adjunction} \cite{ss03} (3.6) if the following two conditions hold.
\begin{enumerate}
\item For any cofibrant objects $X$ and $Y$ in $\M$, the comonoidal structure map \eqref{comonoidal.map} is a weak equivalence in $\N$.
\item For some cofibrant replacement $q : Q\tensorunit^{\M} \to \tensorunit^{\M}$ of the tensor unit in $\M$, the composite
\begin{equation}\label{unit.adjoint}
\nicexy{LQ\tensorunit^{\M} \ar[r]^-{Lq} & L\tensorunit^{\M} \ar[r]^-{\Rbar^0} & \tensorunit^{\N}}
\end{equation}
is a weak equivalence in $\N$, in which $\Rbar_0$ is the adjoint of the structure map $R^0 : \tensorunit^{\M} \to R\tensorunit^{\N}$.
\end{enumerate}
If, furthermore, $L \dashv R$ is a Quillen equivalence, then we call it a \emph{weak (symmetric) monoidal Quillen equivalence}.
\end{definition}

\begin{example} \label{example:ss03}
As discussed in \cite{ss03}, the Dold-Kan correspondence between simplicial modules and non-negatively graded dg modules over a characteristic $0$ field is a weak symmetric monoidal Quillen equivalence.
\end{example}

\subsection{Model Category Assumptions}

\begin{definition}
\label{def:star}
Suppose $\M$ is a symmetric monoidal category and a model category.  Define the following conditions.
\begin{quote}
$(\medstar)$ : Suppose $n \geq 1$, $g : U \to V \in \Msigmanop$ is a weak equivalence with $U$ and $V$ cofibrant in $\M$, and $X \in \Msigman$ is cofibrant in $\M$.  Then the map
\[\nicexy{U \tensoroversigman X \ar[r]^-{g \tensoroversigman X} & V \tensoroversigman X}\]
is a weak equivalence in $\M$.
\end{quote}
\begin{quote}
$(\filledstar)$ : Suppose $n \geq 1$ and $X \in \Msigman$ is cofibrant in $\M$.  Then the coinvariant $X_{\Sigma_n} \in \M$ is also cofibrant.
\end{quote}
\end{definition}

If $\M$ is a model category such that, an object in $\Msigman$ is $\Sigma_n$-projectively cofibrant if and only if it is cofibrant as an object in $\M$, then both $(\medstar)$ and $(\filledstar)$ hold. To verify this claim, observe that this extra condition makes $U$, $V$, and $X$ all $\Sigma_n$-projectively cofibrant above. In this setting, $(\filledstar)$ is true because $(-)_{\Sigma_n}$ is a left Quillen functor, hence takes projectively cofibrant objects to cofibrant objects. To verify $(\medstar)$ for $\M$, note that $(-)\otimes X$ preserves trivial cofibrations in $\Msigman$ \cite[Lemma 2.5.2]{bm06}. Consider the functor $F = (-) \otimes_{\Sigma_n} X = ((-)\otimes X)_{\Sigma_n}: \Msigman \to \M$. Because$(-)_{\Sigma_n}$ is left Quillen, $F$ takes trivial cofibrations between projectively cofibrant objects to weak equivalences, so by Ken Brown's Lemma \cite{hovey} (1.1.12), this functor takes all weak equivalences between projectively cofibrant objects to weak equivalences. Thanks to our assumption about $\M$, this suffices to verify $(\medstar)$. As a consequence of this observation, we have the following examples. Because these examples are rational settings, an object in $\Msigman$ is cofibrant if and only if it is cofibrant as an object in $\M$ \cite[Theorem 8.1.1]{white-yau}.

\begin{example}\label{star.examples}
Fix a field $k$ of characteristic zero. Conditions $(\medstar)$ and $(\filledstar)$ both hold in the model categories of:
\begin{itemize}
\item simplicial $k$-modules;
\item chain complexes, bounded or unbounded, of $k$-modules;
\item non-negatively graded cochain complexes of $k$-modules 
\item the category of $\Fin$-objects in $k$-modules.
\end{itemize}

The last two examples are required for the monoidal dual Dold-Kan correspondence discussed in Section \ref{sec:dualDold}. The category $\Fin$ has the same objects as the cosimplicial category $\Delta$, but its morphisms are all set maps. As discussed in \cite{cortinas},
the category of $\Fin$-objects in $k$-modules is closely related to cosimplicial $k$-modules, and passage from the classical dual Dold-Kan correspondence to the monoidal dual Dold-Kan correspondence does not change the homotopy theory.

\end{example}

The next definition is an equivariant version of Def. \ref{def:weak.symmetric.monoidal}(1).

\begin{definition}
\label{def:sharp}
Suppose $L : \M \adjoint \N : R$ is an adjunction between symmetric monoidal categories that are also model categories with $R$ lax symmetric monoidal.  Define the following condition.
\begin{quote}
$(\#)$ : Suppose $n \geq 1$, $W \in \Msigmanop$ is cofibrant in $\M$, and $X \in \Msigman$ is cofibrant in $\M$.  Then the map
\[\nicexy@C+.5cm{\bigl[L(W \otimes X)\bigr]_{\Sigma_n} \ar[r]^-{(L^2_{W,X})_{\Sigma_n}} & \bigl[LW \otimes LX\bigr]_{\Sigma_n}}\]
is a weak equivalence in $\N$, where $L^2_{W,X}$ is the comonoidal structure map of $L$ \eqref{comonoidal.map}.  
\end{quote}
\end{definition}

Note that $(\#)$ is equivalent to the assertion that the composite
\[\nicexy@C+.5cm{L\left(W \tensoroversigman X\right) \ar[d]_-{\cong} \ar[r] & LW \tensoroversigman LX\\
\bigl[L(W \otimes X)\bigr]_{\Sigma_n} \ar[r]^-{(L^2_{W,X})_{\Sigma_n}} & \bigl[LW \otimes LX\bigr]_{\Sigma_n} \ar[u]_-{=}}\]
is a weak equivalence in $\N$.  The isomorphism on the left comes from the fact that taking $\Sigma_n$-coinvariant is a colimit, and the left adjoint $L$ commutes with colimits.

\begin{example} \label{hash.examples}
Condition $(\#)$ holds when:
\begin{enumerate}
\item $L$ is the identity, in the case of rectification;
\item $L$ is strong symmetric monoidal;
\item  $L \dashv R$ is the Dold-Kan correspondence between simplicial $k$-modules and non-negatively graded chain complexes of $k$-modules for a field $k$ of characteristic zero;
\item $L \dashv R$ is the monoidal dual Dold-Kan correspondence between non-negatively graded cochain complexes of $k$-modules and $\Fin$-objects in $k$-modules, over a field $k$ of characteristic zero.
\end{enumerate}
\end{example}

This is trivial to verify for (1) and (2). For (3) and (4), we use that the model categories in question are in characteristic zero situations, just as in Example \ref{star.examples}. This means $W$ and $X$ are actually projectively cofibrant. The Quillen pair (3) is a \textit{weak monoidal Quillen pair} by \cite{ss03} (4.2, 4.3). This implies, together with the pushout product axiom \cite[Definition 4.2.1]{hovey}, that $L^2_{W,X}$ is a weak equivalence between projectively cofibrant objects. Now $(\#)$ follows from Ken Brown's lemma, applied to the left Quillen functor $(-)_{\Sigma_n}$. Lastly, $(\#)$ holds for (4) because the left adjoint is strong symmetric monoidal, as discussed in \cite{cortinas}.

\section{Colored Operads}
\label{sec:colored}

In this section we recall some results regarding colored operads that will be needed in later sections.

\subsection{Colors and Profiles}

Here we recall from \cite{jy2} some notations regarding colors that are needed to talk about colored objects.

\begin{definition}[Colored Objects]
\label{def:profiles}
Fix a set $\fC$, whose elements are called \textbf{colors}.
\begin{enumerate}
\item
A \emph{$\fC$-profile} is a finite sequence of elements in $\fC$, say,
\[\uc = (c_1, \ldots, c_m) = c_{[1,m]}\]
with each $c_i \in \fC$.  If $\fC$ is clear from the context, then we simply say \emph{profile}. The empty $\fC$-profile is denoted $\emptyset$, which is not to be confused with the initial object in $\calm$.  Write $|\uc|=m$ for the \emph{length} of a profile $\uc$.
\item
An object in the product category $\prod_{\fC} \calm = \calm^{\fC}$ is called a \emph{$\fC$-colored object in $\calm$}, and similarly for a map of $\fC$-colored objects.  A typical $\fC$-colored object $X$ is also written as $\{X_a\}$ with $X_a \in \calm$ for each color $a \in \fC$.
\item
Suppose $X \in \calmc$ and $c \in \fC$.  Then $X$ is said to be \emph{concentrated in the color $c$} if $X_d = \varnothing$ for all $c \not= d \in \fC$.
\item
Suppose $f : X \to Y \in \calm$ and $c \in \fC$.  Then $f$ is said to be \emph{concentrated in the color $c$} if both $X$ and $Y$ are concentrated in the color $c$.
\end{enumerate}
\end{definition}

Next we define the colored version of a $\Sigma$-object, also known as a symmetric sequence.

\begin{definition}[Colored Symmetric Sequences]
\label{def:colored-sigma-object}
Fix a set $\fC$ of colors.
\begin{enumerate}
\item
If $\ua = (a_1,\ldots,a_m)$ and $\ub$ are $\fC$-profiles, then a \emph{map} (or \emph{left permutation}) $\sigma : \ua \to \ub$ is a permutation $\sigma \in \Sigma_{|\ua|}$ such that
\[\sigma\ua = (a_{\sigma^{-1}(1)}, \ldots , a_{\sigma^{-1}(m)}) = \ub\]
This necessarily implies $|\ua| = |\ub| = m$.
\item
The \emph{groupoid of $\fC$-profiles}, with left permutations as the isomorphisms, is denoted by $\pofc$.  The opposite groupoid $\pofcop$ is regarded as the groupoid of $\fC$-profiles with \emph{right permutations}
\[\ua\sigma = (a_{\sigma(1)}, \ldots , a_{\sigma(m)})\]
as isomorphisms.
\item
The \emph{orbit} of a profile $\ua$ is denoted by $[\ua]$.  The maximal connected sub-groupoid of $\pofc$ containing $\ua$ is written as $\sigmabra$.  Its objects are the left permutations of $\ua$.  There is a decomposition
\begin{equation}
\label{pofcdecomp}
\pofc \cong \coprod_{[\ua] \in \pofc} \sigmabra,
\end{equation}
where there is one coproduct summand for each orbit $[\ua]$ of a $\fC$-profile.  By $[\ua] \in \pofc$ we mean that $[\ua]$ is an orbit in $\pofc$.
\item
Define the diagram category
\[\symseqcm = \calm^{\sigmacopc},\]
whose objects are called \emph{$\fC$-colored symmetric sequences}.  By the decomposition \eqref{pofcdecomp}, there is a decomposition
\[\symseqcm \cong 
\prod_{\dbrch \in \sigmacopc} \calm^{\sigmabrcopd},\]
where $\sigmabrcopd \cong \sigmaopc$.  
\item
For $X \in \symseqcm$, we write
\[X\singledbrc \in \calm^{\sigmabrcopd} \cong \calm^{\sigmabrcop}\]
for its $\dbrch$-component.  For $(\uc;d) \in \sigmacopc$ (i.e., $\uc$ is a $\fC$-profile and $d \in \fC$), we write
\[X\duc \in \calm\]
for the value of $X$ at $(\uc;d)$.
\end{enumerate}
\end{definition}

\begin{remark}
\label{soneobject}
In the one-colored case (i.e., $\fC = \{*\}$), for each integer $n \geq 0$, there is a unique $\fC$-profile of length $n$, usually denoted by $[n]$.  We have $\Sigma_{[n]} = \Sigma_n$, the symmetric group $\Sigma_n$ regarded as a one-object groupoid.  So we have
\[
\pofc = \coprod_{n \geq 0} \Sigma_n = \Sigma
\andspace
\symseqcm = \calm^{\sigmacopc} = \calm^{\sigmaop}.
\]
In other words, one-colored symmetric sequences are symmetric sequences (also known as $\Sigma$-objects and collections) in the usual sense.
\end{remark}

From now on, assume that $\fC$ is a fixed non-empty set of colors, unless otherwise specified.

\subsection{Colored Circle Product}

We will define $\fC$-colored operads as monoids with respect to the $\fC$-colored circle product.  To define the latter, we need the following definition.

\begin{definition}[Tensored over a Category]
\label{def:tensorover}
Suppose $\cald$ is a small groupoid, $X \in \calm^{\caldop}$, and $Y \in \calm^{\cald}$.  Define the object $X \otimes_{\cald} Y \in \calm$ as the colimit of the composite
\[
\nicexy{
\cald \ar[r]^-{\cong \Delta} 
& \caldop \times \cald \ar[r]^-{(X,Y)}
& \calm \times \calm \ar[r]^-{\otimes}
& \calm,
}\]
where the first map is the diagonal map followed by the isomorphism $\cald \times \cald \cong \caldop \times \cald$. The colimit $X \otimes_{\cald} Y$ coincides with the coend $\int^{d\in \cald} X_d \otimes Y_d$.
\end{definition}

We will mainly use the construction $\otimes_{\cald}$ when $\cald$ is the finite connected groupoid $\sigmabrc$ for some orbit $[\uc] \in \pofc$.

\begin{convention}
For an object $A \in \calm$, $A^{\otimes 0}$ is taken to mean $\tensorunit$, the $\otimes$-unit in $\calm$.
\end{convention}

\begin{definition}[Colored Circle Product]
\label{def:colored-circle-product}
Suppose $X,Y  \in \symseqcm$, $d \in \fC$, $\uc = (c_1,\ldots,c_m) \in \pofc$, and $[\ub] \in \pofc$ is an orbit.
\begin{enumerate}
\item
Define the object
\[Y^{\uc} \in \calm^{\pofcop} \cong \prod_{[\ub] \in \pofc} \calm^{\sigmabrbop}\]
as having the $[\ub]$-component
\[Y^{\uc}([\ub]) =
\coprod_{\substack{\{[\ub_j] \in \pofc\}_{1 \leq j \leq m} \,\mathrm{s.t.} \\
[\ub] = [(\ub_1,\ldots,\ub_m)]}} 
\Kan^{\sigmabrbop} 
\left[\bigotimes_{j=1}^m Y \cjbrbj\right] 
\in \calm^{\sigmabrbop}.
\]
The above left Kan extension is defined as
\[\nicexy{\prod_{j=1}^m \sigmabrbjop 
\ar[d]_-{\mathrm{concatenation}} 
\ar[rr]^-{\prod Y \binom{c_j}{-}} 
&& \calm^{\times m} \ar[d]^-{\otimes}
\\
\sigmabrbop \ar[rr]_-{\Kan^{\sigmabrbop}\left[\otimes Y(\vdots)\right]}^-{\mathrm{left ~Kan~ extension}}  && \calm.}\]
\item
By allowing left permutations of $\uc$ above, we obtain
\[
Y^{[\uc]} \in \calm^{\pofcop \times \sigmabrc} \cong \prod_{[\ub] \in \pofc} \calm^{\sigmabrbop \times \sigmabrc}\]
with components
\[Y^{[\uc]}([\ub]) \in \calm^{\sigmabrbop \times \sigmabrc}.\]
\item 
The \emph{$\fC$-colored circle product}
\[X \circ Y \in \symseqcm\]
is defined to have components
\[(X \circ Y)\singledbrb 
= \coprod_{[\uc] \in \pofc} 
X\singledbrc \tensorover{\sigmabrc} Y^{[\uc]}([\ub]) \in \calm^{\sigmaopb \times \{d\}},
\]
where the coproduct is indexed by all the orbits in $\pofc$, as $d$ runs through $\fC$ and $[\ub]$ runs through all the orbits in $\pofc$. The functor $X\circ -$ restricts to a functor on $\calmc$: given a $\fC$-colored object $A$, the $d$-colored entry of $X \circ A$ is equal to the coend

\[
(X \circ A)_d = \dint^{\uc \in \Sigma_{C}} X\duc \otimes A_{\uc}.
\]

Inside this coend, $X\duc$ is contravariant in $\uc$ and $A_{\uc}$ is covariant in $\uc$.
\end{enumerate}
\end{definition}

\begin{proposition}[$=$ 3.2.18 in \cite{white-yau}]
\label{circle-product-monoidal}
With respect to the $\fC$-colored circle product $\circ$, $\symseqcm$ is a monoidal category.
\end{proposition}

\begin{definition}
\label{def:colored-operad}
For a set $\fC$ of colors, a \emph{$\fC$-colored operad} in $\calm$ is a monoid \cite{maclane} (VII.3) in the monoidal category $(\symseqcm, \comp)$.
\end{definition}

\subsection{Algebras over Colored Operads}

\begin{definition}
\label{colored-operad-algebra}
Suppose $\sO$ is a $\fC$-colored operad.  The category of algebras over the monad \cite{maclane} (VI.2)
\[\sO \comp - : \calmc \to \calmc\]
is denoted by $\alg(\sO)$, whose objects are called \emph{$\sO$-algebras} in $\calm$.
\end{definition}

\begin{definition}
\label{def:asubc}
Suppose $A = \{A_c\}_{c\in \fC} \in \calm^{\fC}$ and $\uc = (c_1,\ldots,c_n) \in \pofc$ with orbit $[\uc]$.  Define the object
\[
A_{\uc} = \bigotimes_{i=1}^n A_{c_i} = A_{c_1} \otimes \cdots \otimes A_{c_n} \in \calm
\]
and the diagram $A_{\smallbrc} \in \calm^{\sigmabrc}$ with values
\[A_{\smallbrc}(\uc') = A_{\uc'}\]
for each $\uc' \in [\uc]$.  All the structure maps in the diagram $A_{\smallbrc}$ are given by permuting the factors in $A_{\uc}$.
\end{definition}

There is a free-forgetful adjoint pair
\[\nicexy{\calmc \ar@<2pt>[r]^-{\sO \comp -} 
& \alg(\sO) \ar@<2pt>[l]}\]
for each $\fC$-colored operad $\sO$.

\begin{proposition}[$=$ 4.2.1 in \cite{white-yau}]
\label{algebra-bicomplete}
Suppose $\sO$ is a $\fC$-colored operad in $\M$.  Then the category $\alg(\sO)$ has all small limits and colimits, with reflexive coequalizers and filtered colimits preserved and created by the forgetful functor $\alg(\sO) \to \calmc$.
\end{proposition}

\begin{definition}\label{lifted.right.adjoint}
Suppose $L : \M \adjoint \N : R$ is an adjunction between symmetric monoidal categories with $R$ lax symmetric monoidal.  Suppose $f : \sO \to R\sP$ is a map of $\fC$-colored operads in $\M$ with $\fC$ a set, $\sO$ a $\fC$-colored operad in $\M$, and  $\sP$ a $\fC$-colored operad in $\N$.  Here $R$ is applied entrywise to $\sP$;  according to \cite{jy2} (Theorem 12.11) $R\sP$ is a $\fC$-colored opeard in $\M$.  
\begin{enumerate}
\item Define an induced functor
\begin{equation}\label{R.alg}
\nicexy{\algo &  \algp \ar[l]_-{R}}
\end{equation}
as follows.  For a $\sP$-algebra $A$, apply $R$ entrywise to $A \in \ntoc$ to obtain $RA \in \mtoc$.  Then $RA$ becomes an $\sO$-algebra with structure map the composite
\[\nicexy{\sO\duc \otimes (RA)_{\uc} \ar[r] \ar[d]_-{(f,\Id)} & RA_d\\
R\sP\duc \otimes (RA)_{\uc} \ar[r]^-{R^2} & R\bigl(\sP\duc \otimes A_{\uc}\bigr) \ar[u]_-{R\lambda}}\]
for $d \in \fC$ and $\uc \in \pofc$.  The map
\[\nicexy{\sP\duc \otimes A_{\uc} \ar[r]^-{\lambda} & A_d}\]
is a $\sP$-algebra structure map of $A$, and $R^2$ is a repeated application of the lax monoidal structure map of $R$.  Note that if $R = \Id$, then the functor $R : \algp \to \algo$ is the restriction along $f$.
\item
By the Adjoint Lifting Theorem \cite{borceux} (4.5.6), the functor $R$ \eqref{R.alg} admits a left adjoint $\Lbar : \algo \to \algp$, which is in general \emph{not} $L$ entrywise.  The original adjoint pair $L \dashv R$ is related to the lifted adjoint pair $\Lbar \dashv R$ as follows:
\begin{equation}\label{lbar.ocomp.diagram}
\nicexy{\algo \ar@<2.5pt>[r]^-{\Lbar} \ar@<2.5pt>[d]^-{U}
& \algp \ar@<2.5pt>[l]^-{R} \ar@<2.5pt>[d]^-{U} \\
\mtoc \ar@<2.5pt>[r]^-{L} \ar@<2.5pt>[u]^-{\sO \comp -}  
& \ntoc \ar@<2.5pt>[l]^-{R} \ar@<2.5pt>[u]^-{\sP \comp -}}
\end{equation}
On each side, the vertical arrows form a free-forgetful adjunction. The right adjoint diagram is commutative, i.e., $UR = RU$.  By uniqueness the left adjoin diagram is also commutative, i.e., 
\begin{equation}\label{lbar.ocomp}
\Lbar\left(\sO \comp X\right) = \sP \comp (LX)
\end{equation}
in which $LX$ means $L$ is applied entrywise to $X \in \mtoc$.
\item For each $\sO$-algebra $A$, the unit $A \to R\Lbar A \in \algo$ of the adjunction, when regarded as a map in $\mtoc$, has an entrywise adjoint
\begin{equation}\label{comparison.map}
\nicexy{LUA \ar[r]^-{\chi_A} & U\Lbar A \in \ntoc}
\end{equation}
called the \emph{comparison map}.  We will usually write the comparison map as $LA \to \Lbar A$, omitting the forgetful functors from the notation. 
\end{enumerate}
\end{definition}

\begin{example}\label{comparison.initial}
For the initial $\sO$-algebra $\initialo = \left\{\sO\singledempty\right\}_{d \in \fC}$, we have
\[\Lbar\initialo = \initialp = \left\{\sP\singledempty\right\}_{d \in \fC}\]
because $\Lbar$ is a left adjoint. In this case, the comparison map $\chi_{\initialo}$ is entrywise the adjoint of $f$,
\[\nicexy{(L\initialo)_d = L\sO\singledempty \ar[r]^-{\fbar} & \sP\singledempty = (\Lbar\initialo)_d \in \N}\]
for $d \in \fC$.
\end{example}

\begin{example}\label{comparison.adjoint}
Suppose $A \in \algo$, $B \in \algp$, and $\varphi : \Lbar A \to B \in \algp$.  Then the adjoint of $\varphi$ is a map $\varphibar : A \to RB \in \algo$.  Considering $\varphibar$ as a map in $\mtoc$, its entrywise adjoint is the composite
\[\nicexy{LA \ar[r]^-{\chi_A} & \Lbar A \ar[r]^-{\varphi} & B \in \ntoc.}\]
This example will be important in the proof of Theorem \ref{main.theorem}.  In the cases of monoids and of $1$-colored non-symmetric operads, the comparison map appeared in \cite{ss03} (5.1) and \cite{muro14} (7-2), respectively. In the situation of rectification, when $L = R = Id$, $\Lbar = \varphi_!$ and $R = \varphi^*$.

\end{example}

\subsection{Filtration for Pushouts of Colored Operadic Algebras}

\begin{definition}
Suppose $X \in \symseqcm$, $d \in \fC$, and $[\ua], [\ub], [\uc]$ are orbits in $\sigmaofc$.  Define the diagram
\[X \singledbrabrc \in \calm^{\sigmabraop \times \sigmabrcop \times \{d\}}\]
as having the objects
\[X \singledbrabrc(\ua'; \uc') 
= X \singledaprimecprime \in \calm\]
for $\ua' \in [\ua]$ and $\uc' \in [\uc]$ and the structure maps of $X$.
\end{definition}

\begin{definition}[$\sO_A$ for $\sO$-algebras]
\label{oaalgebra}
Suppose $\sO$ is a $\fC$-colored operad and $A \in \alg(\sO)$.  Define $\sO_A \in \symseqcm$ as follows.  For $d \in \fC$ and orbit $[\uc] \in \sigmaofc$, define the component
\[\sO_A\singledbrc \in \calm^{\sigmabrcop \times \{d\}}\]
as the reflexive coequalizer of the diagram
\[\nicexy{\coprod\limits_{[\ua] \in \sigmaofc} 
\sO\singledbrabrc \tensorover{\sigmabra} (\sO \circ A)_{\smallbra}
\ar@<-3pt>[r]_-{d_0} \ar@<3pt>[r]^-{d_1}
& \coprod\limits_{[\ua] \in \sigmaofc} 
\sO\singledbrabrc \tensorover{\sigmabra} A_{\smallbra}
\ar@/_1.5pc/[l]}
\]
with
\begin{itemize}
\item
the coequalizer taken in $\calm^{\sigmabrcop \times \{d\}}$, 
\item
$d_0$ induced by the operadic composition on $\sO$,
\item
$d_1$ induced by the $\sO$-algebra action on $A$, and 
\item
the common section induced by $A \cong \scrI \circ A \to \sO \circ A$, where $\scrI$ is the unit for the $\fC$-colored circle product.
\end{itemize}
\end{definition}

\begin{proposition}[$=$ 5.1.1 in \cite{white-yau}]
\label{o-sub-empty}
Suppose $\sO$ is a $\fC$-colored operad, and $\varnothing$ is the initial $\sO$-algebra.  Then there is an isomorphism
\[\sO_{\varnothing} \cong \sO\]
in $\symseqcm$.
\end{proposition}

\begin{proposition}[$=$ 5.1.3 in \cite{white-yau}]
\label{osuba-empty}
Suppose $\sO$ is a $\fC$-colored operad, $A \in \alg(\sO)$, and $d \in \fC$.  Then there is a natural isomorphism
\[\sO_A\singledempty \cong A_d\]
in $\calm$.
\end{proposition}

\begin{lemma}[$=$ 5.3.1 in \cite{white-yau}]
\label{osubdc}
Suppose $\sO$ is a $\fC$-colored operad, $d \in \fC$, and $[\uc] \in \sigmaofc$.  Then the functor
\[\sO_{(-)}\singledbrc : \alg(\sO) \to \calm^{\sigmabrcop \times \{d\}}\]
preserves reflexive coequalizers and filtered colimits.
\end{lemma}

\begin{definition}
\label{def:q-construction}
Suppose $i : X \to Y \in \calmc$ is concentrated at a single color $c \in \fC$ (so $X_b = Y_b = \varnothing$ whenever $b \not= c$) and $t \geq 1$. Define $i^{\boxprod t}: Q_{t-1}^t \to Y^{\otimes t}$ to be the $t$-fold iterated pushout product of $i$ with itself. The object $Q_{t-1}^t$ is the colimit of the punctured $t$-dimensional cube formed by words of length $t$ in the symbols $X$ and $Y$, but omitting the word $Y^{\otimes t}$. The morphism $i^{\boxprod t}$ has a natural action of $\Sigma_t$ by permuting the letters in the words.
\end{definition}

The object $Q_{t-1}^t$ is part of a filtration of $Y^{\otimes t}$ \cite{white-yau} (4.3.15)  that was essential to the proof of the following two results.

\begin{proposition}[$=$ 4.3.16 in \cite{white-yau}]
\label{free-pushout-filtration}
Suppose $\sO$ is a $\fC$-colored operad, $A \in \alg(\sO)$, $i : X \to Y \in \calmc$ is concentrated at a single color $c \in \fC$, and
\[\nicexy{\sO \circ X \ar[d]_-{\id \circ i} \ar[r]^-{f} 
& A \ar[d]^-{j} \\
\sO \circ Y \ar[r] & A_{\infty}}
\]
is a pushout in $\alg(\sO)$.  Then the map $j \in \mtoc$ factors naturally into a countable composition
\[\narrowxy{A = A_0 \ar[r]^-{j_1} & A_1 \ar[r]^-{j_2} & A_2 \ar[r]^-{j_3} & \cdots \ar[r] & A_{\infty} \in \mtoc}\]
such that, for each color $d \in \fC$ and $t \geq 1$, the $d$-colored entry of $j_t$ is inductively defined as the pushout in $\calm$
\begin{equation}
\label{one-colored-jt-pushout}
\nicexy{\sO_A \singledbrtc \tensorover{\Sigma_t} Q^{t}_{t-1}
\ar[d]_-{\id \tensorover{\Sigma_t} i^{\boxprod t}} \ar[r]^-{f^{t-1}_*} 
& (A_{t-1})_d \ar[d]^-{j_t} \\
\sO_A \singledbrtc \tensorover{\Sigma_t} Y^{\otimes t} \ar[r]_-{\xi_{t}} 
&  (A_t)_d}
\end{equation}
with $f^{t-1}_*$ induced by $f$ and $tc = (c,\ldots,c)$ with $t$ copies of $c$.
\end{proposition}

\begin{proposition}[$=$ 5.3.2 in \cite{white-yau}]
\label{o-ainfinity}
Suppose $\sO$ is a $\fC$-colored operad, $A \in \alg(\sO)$, $i : X \to Y \in \calmc$ is concentrated in one color $b \in \fC$, and
\[\nicexy{\sO \circ X \ar[d]_-{\id \circ i} \ar[r]^-{f} & A \ar[d]^-{j}
\\ \sO \circ Y \ar[r] & A_{\infty}}
\]
is a pushout in $\alg(\sO)$.  Suppose $d \in \fC$ and $[\uc] \in \sigmaofc$.  Then the map $\sO_{j} \in \calm^{\sigmabrcop \times \{d\}}$ factors naturally into a countable composition
\[\narrowxy{\osuba\singledbrc = \osubazero \singledbrc \ar[r]^-{j_1}
& \osubaone \singledbrc \ar[r]^-{j_2}
& \osubatwo \singledbrc \ar[r]^-{j_3}
& \cdots \ar[r] & \osubainfinity \singledbrc}\]
in $\calm^{\sigmabrcop \times \{d\}}$ in which each $j_t$ for $t \geq 1$ is defined inductively as the pushout
\begin{equation}
\label{osubjt}
\nicexy{
\osuba \singledbrtbbrc \tensorover{\Sigma_t} Q^{t}_{t-1} 
\ar[d]_-{\id \tensorover{\Sigma_t} i^{\boxprod t}} \ar[r]^-{f_*}
&  \osubatminusone\singledbrc \ar[d]^-{j_t}
\\ \osuba \singledbrtbbrc \tensorover{\Sigma_t} Y^{\otimes t}
\ar[r]^-{\xi_t} & \osubat \singledbrc}
\end{equation}
in $ \calm^{\sigmabrcop \times \{d\}}$, where $tb = (b,\ldots,b)$ with $t$ copies of $b$.
\end{proposition}

\subsection{Semi-Model Structures on Algebras over Entrywise Cofibrant Operads}

\begin{definition}
\label{def:club}
Suppose $\calm$ is a symmetric monoidal category and is a model category.  Define the following condition.
\begin{quote}
$(\clubsuit)$ : For each $n \geq 1$ and $X \in \calm^{\sigmaop_n}$ that is cofibrant in $\calm$, if $f$ is a (trivial) cofibration, then so is
$X \tensorover{\Sigma_n} f^{\boxprod n}.$
\end{quote}
The condition $(\clubsuit)$ for cofibrations will be referred to as $\clubcof$, and the condition for trivial cofibrations as $\clubacof$.  So $(\clubsuit) = \clubcof + \clubacof$.
\end{definition}

\begin{example}
As in Example \ref{star.examples}, condition $(\clubsuit)$ holds whenever cofibrancy in $\Msigman$ coincides with cofibrancy in $\M$.  In particular, it holds in:
\begin{itemize}
\item simplicial modules over a characteristic $0$ field;
\item chain complexes, bounded or unbounded, over a characteristic $0$ field
\end{itemize}
\end{example}

\begin{theorem}[$=$ 6.2.3 in \cite{white-yau}]
\label{theorem623}
Suppose 
$\calm$ is a cofibrantly generated monoidal model category satisfying $(\clubsuit)$.  Then for each entrywise cofibrant $\fC$-colored operad $\sO$ in $\calm$, the category $\alg(\sO)$ admits a cofibrantly generated \textbf{semi}-model structure over $\mtoc$  such that the weak equivalences and fibrations are created in $\calm$.  Moreover:
\begin{enumerate}
\item
If $j : A \to B \in \alg(\sO)$ is a cofibration with $A$ cofibrant in $\alg(\sO)$, then the underlying map of $j$ is entrywise a cofibration.
\item
Every cofibrant $\sO$-algebra is entrywise cofibrant in $\calm$.
\end{enumerate}
\end{theorem}

\begin{lemma}[$=$ 6.2.4 in \cite{white-yau}]
\label{middle-row-lemma}
Suppose $\calm$ is a symmetric monoidal closed category and is a model category satisfying $\clubcof$, and $\sO$ is a $\fC$-colored operad in $\calm$.
\begin{enumerate}
\item
Suppose $j : A \to B \in \alg(\sO)$ is an $(\sO \comp \calm_{\cof})$-cofibration, i.e., a retract of a transfinite composition of pushouts of maps in $\sO \comp \calm_{\cof}$.  Suppose also that $\sO_A$ is entrywise cofibrant in $\calm$.  Then $\sO_A \to \sO_B$ is entrywise a cofibration in $\calm$.
\item
Suppose $\sO$ is entrywise cofibrant in $\calm$, and suppose $\varnothing \to A \in \alg(\sO)$ is a $(\sO \comp \calm_{\cof})$-cofibration.  Then $\sO_A$ is entrywise cofibrant in $\calm$.
\end{enumerate}
\end{lemma}

\begin{definition}\label{def:nice.qeq}
A \emph{nice Quillen equivalence} $L : \M \adjoint \N : R$ is a weak symmetric monoidal Quillen equivalence (Def. \ref{def:weak.symmetric.monoidal}) between cofibrantly generated monoidal model categories such that the following conditions hold.
\begin{enumerate}
\item $(\filledstar)$ (Def. \ref{def:star}), $(\#)$ (Def. \ref{def:sharp}), and $(\clubsuit)$ (Def. \ref{def:club}) hold in $\M$ and $\N$.  
\item $\N$ satisfies $(\medstar)$ (Def. \ref{def:star}).
\item Every generating cofibration in $\M$ has cofibrant domain.
\end{enumerate}
\end{definition}

\begin{example} \label{ex:nice}
Let $k$ be a field of characteristic zero. Then the following are examples of nice Quillen equivalences:
\begin{enumerate}
\item the Dold-Kan correspondence between simplicial $k$-modules and non-negatively graded chain complexes of $k$-modules;
\item the monoidal dual Dold-Kan correspondence $Q: \Ch(k)^{\geq 0} \leftrightarrows (\Vect_k)^{\Fin}: P$, between non-negatively graded cochain complexes of $k$-modules and the category of $\Fin$-objects in $k$-vector spaces, which is closely related to cosimplicial $k$-vector spaces.
\end{enumerate}

The monoidal dual Dold-Kan correspondence \cite{cortinas} is recalled in Section \ref{sec:dualDold}. The authors verified $(\clubsuit)$ (as a consequence of a stronger hypothesis, $(\spadesuit)$) for chain complexes over a field of characteristic zero in \cite{white-yau} (Theorem 8.1.1). The proof for simplicial $k$-modules, for $\Ch(k)^{\geq 0}$, and for $(\Vect_k)^{\Fin}$ are analogous. The cofibrancy hypothesis \ref{def:nice.qeq}(3) is trivial because every object in these categories is cofibrant. We verified $(\filledstar), (\#)$, and $(\medstar)$ in \ref{star.examples} and \ref{hash.examples}.
\end{example}

\section{Main Result} \label{sec:main}

\subsection{Key Step}

By Proposition \ref{osuba-empty}, for each color $d \in \fC$, there are natural isomorphisms
\[L\sO_A\dnothing \cong LA_d \andspace \sP_{\Lbar A}\dnothing = (\Lbar A)_d.\]
 A key part of the proof of Theorem \ref{main.theorem} is the following result.

\begin{proposition}
\label{loa.to.pla}
Suppose:
\begin{enumerate}
\item $L : \M \adjoint \N : R$ is a nice Quillen equivalence (Def. \ref{def:nice.qeq}).
\item $f : \sO \to R\sP$ is a map of $\fC$-colored operads in $\M$ with $\fC$ a set, $\sO$ an entrywise cofibrant $\fC$-colored operad in $\M$, and $\sP$ an entrywise cofibrant $\fC$-colored operad in $\N$.  The entrywise adjoint $\fbar : L\sO \to \sP$ is an entrywise weak equivalence in $\N$. 
\end{enumerate}
Suppose $A$ is a cofibrant $\sO$-algebra.  Then the map $f : \sO \to R\sP$ induces a natural entrywise weak equivalence
\[\nicexy{L\sO_A \ar[r]^-{f_{\infty}} & \sP_{\Lbar A} \in \symseqcn}\]
whose value at $\dnothing$ is the comparison map $\chi_A : LA \to \Lbar A$ \eqref{comparison.map} evaluated at $d$ for each $d \in \fC$.
\end{proposition}

Below, we will show that, if $\sO$ and $\sP$ are $\Sigma$-cofibrant, then we can weaken assumption (1) above to only requiring that $(L,R)$ be a weak symmetric monoidal Quillen equivalence and that the domains of the generating cofibrations in $\M$ be cofibrant.

\begin{proof}
The generating cofibrations in $\algo$ have the form $\sO \comp i$ for some generating cofibration $i$ in $\M$, regarded as a map in $\mtoc$ concentrated in a single color.  Each cofibration in $\algo$ is a retract of a transfinite composition of pushouts of generating cofibrations.  By a retract argument we may assume that the map $\initialo \to A \in \algo$ is a transfinite composition
\begin{equation}\label{a.cell.complex}
\narrowxy{\initialo = A^0 \ar[r] & A^1 \ar[r] & A^2 \ar[r] & \cdots \ar[r] & A \in \algo}
\end{equation}
such that, for each $t \geq 1$, the map $A^{t-1} \to A^t$ is a pushout
\begin{equation}\label{at.pushout}
\nicexy{\sO \comp X \ar[d]_-{\Id \comp i} \ar[r] & A^{t-1} \ar[d]\\
\sO \comp Y \ar[r] & A^t}
\end{equation}
in $\algo$ for some generating cofibration $i : X \to Y$ in $\M$, regarded as a map in $\mtoc$ concentrated in a single color, say, $c \in \fC$.  Both the map $i$ and the color $c$ depend on the index $t$.  Note that, since the initial $\sO$-algebra is cofibrant and $A^{t-1} \to A^t \in \algo$ is a cofibration, all the $A^t$ are cofibrant $\sO$-algebras.  Moreover, by assumption on $\M$, the generating cofibration $i$ is a cofibration between cofibrant objects.

We apply $\sO_{(-)}$ (Def. \ref{oaalgebra}) to the transfinite composition \eqref{a.cell.complex} and use Proposition \ref{o-sub-empty} on $A^0$ and Lemma \ref{osubdc} on the colimit.  We obtain the transfinite composition
\begin{equation}\label{osuba.cell}
\narrowxy{\sO \cong \sO_{A^0} \ar[r] & \sO_{A^1} \ar[r] & \sO_{A^2} \ar[r] & \cdots \ar[r] & \sO_{A} \in \symseqcm.}
\end{equation}
Since $\sO$ is entrywise cofibrant and since all the $A^t$ are cofibrant $\sO$-algebras, by Lemma \ref{middle-row-lemma}, in \eqref{osuba.cell} every map is an entrywise cofibration between entrywise cofibrant objects in $\M$.  Applying the left Quillen equivalence \cite{hirschhorn} (11.6.5(2))
\[\nicexy{L : \symseqcm \ar[r] & \symseqcn}\]
to \eqref{osuba.cell}, we obtain the transfinite composition
\begin{equation}\label{losuba.cell}
\narrowxy{L\sO \cong L\sO_{A^0} \ar[r] & L\sO_{A^1} \ar[r] & L\sO_{A^2} \ar[r] & \cdots \ar[r] & L\sO_{A} \in \symseqcn}
\end{equation}
of entrywise cofibrations between entrywise cofibrant objects in $\N$.

Next we apply the left adjoint $\Lbar : \algo \to \algp$ to the transfinite composition \eqref{a.cell.complex} and the pushouts \eqref{at.pushout}.  We obtain the transfinite composition
\begin{equation}\label{lbar.star}
\narrowxy{\initialp = \Lbar\initialo = \Lbar A^0 \ar[r] & \Lbar A^1 \ar[r] & \Lbar A^2 \ar[r] & \cdots \ar[r] & \Lbar A \in \algp}
\end{equation}
such that, for each $t \geq 1$, the map $\Lbar A^{t-1} \to \Lbar A^t$ is a pushout
\begin{equation}\label{lbar.at.pushout}
\nicexy{\sP \comp (LX) = \Lbar(\sO \comp X) \ar[d]_-{\Id \comp Li}^-{=\, \Lbar(\Id \comp i)} \ar[r] & \Lbar A^{t-1} \ar[d]\\ \sP \comp (LY) = \Lbar(\sO \comp Y) \ar[r] & \Lbar A^t}
\end{equation}
in $\algp$. The equalities on the left come from \eqref{lbar.ocomp}.  Since $\sP$ is also entrywise cofibrant, similar to the paragraph containing \eqref{osuba.cell}, applying $\sP_{(-)}$ to the transfinite composition \eqref{lbar.star} yields the transfinite composition
\begin{equation}\label{psub.lbar.star}
\narrowxy{\sP \cong \sP_{\Lbar A^0} \ar[r] & \sP_{\Lbar A^1} \ar[r] & \sP_{\Lbar A^2} \ar[r] & \cdots \ar[r] & \sP_{\Lbar A} \in \symseqcn}
\end{equation}
of entrywise cofibrations between entrywise cofibrant objects in $\N$.

Consider the commutative ladder diagram from \eqref{losuba.cell} to \eqref{psub.lbar.star},
\begin{equation}\label{lostar.to.plbarstar}
\nicexy{L\sO \cong L\sO_{A^0} \ar[d]_-{f_0 \defn \fbar}\ar[r] & L\sO_{A^1} \ar[d]_-{f_1} \ar[r] & L\sO_{A^2} \ar[d]_-{f_2} \ar[r] & \cdots \ar[r] & L\sO_{A} \ar[d]_-{\colim\, f_t \,=}^-{ f_{\infty}}\\
 \sP \cong \sP_{\Lbar A^0} \ar[r] & \sP_{\Lbar A^1} \ar[r] & \sP_{\Lbar A^2} \ar[r]  & \cdots \ar[r] & \sP_{\Lbar A}}
\end{equation}
in $\symseqcn$, in which $f_0$ is defined to be $\fbar : L\sO \to \sP$.  Our goal is to show that the colimit $f_{\infty}$ is a weak equivalence, i.e., entrywise weak equivalence in $\N$.  By \cite{hirschhorn} (17.9.1) it suffices to show by induction that each vertical map $f_t$ for $t \geq 0$ is a weak equivalence.  The initial map $f_0 = \fbar$ is a weak equivalence by assumption.

For the induction step, suppose $t \geq 1$ and that the map
\[\nicexy{L\sO_{A^{t-1}} \ar[r]^-{f_{t-1}} & \sP_{\Lbar A^{t-1}} \in \symseqcn}\]
is a weak equivalence.  We want to show that $f_t$ is a weak equivalence.  The map $f_t$ is inductively defined as follows.  Pick $d \in \fC$ and $\ub \in \profc$.  Applying Proposition \ref{o-ainfinity} to the pushout \eqref{at.pushout}, we see that the map
\[\sO_{A^{t-1}}\dbrb \to \sO_{A^t}\dbrb\]
is a countable composition
\begin{equation}\label{osubstar.t}
\narrowxy{\sO_{A^{t-1}}\dbrb = \sO_{A^{t-1}}^0\dbrb \ar[r] & \sO_{A^{t-1}}^1\dbrb \ar[r] & \sO_{A^{t-1}}^2\dbrb \ar[r] & \cdots \ar[r] & \sO_{A^t}\dbrb}
\end{equation}
in $\Msigmabopd$ in which, for each $r \geq 1$, the $r$th map is the pushout
\begin{equation}\label{star.r.t}
\nicexy{\sO_{A^{t-1}}\drcub \tensoroversigmar Q^r_{r-1}(i) \ar[r] \ar[d]_{\Id \tensoroversigmar i^{\boxprod r}} & \sO_{A^{t-1}}^{r-1}\dbrb \ar[d]\\
\sO_{A^{t-1}}\drcub \tensoroversigmar Y^{\otimes r} \ar[r] & \sO^r_{A^{t-1}}\dbrb}
\end{equation}
in $\Msigmabopd$.  In the previous diagram, $rc = (c, \ldots, c)$ with $r$ copies of the color $c \in \fC$, and the top horizontal map is naturally induced by the map $\sO \comp X \to A^{t-1}$.  We already observed above that every $\sO_{A^k}$ is entrywise cofibrant in $\M$.  Since $i : X \to Y$ is a cofibration in $\M$, by $\clubcof$ (Def. \ref{def:club}) the left vertical map in \eqref{star.r.t} is entrywise a cofibration in $\M$, hence so is the right vertical map.  An induction shows that in \eqref{osubstar.t} every map is an entrywise cofibration between entrywise cofibrant objects in $\M$.  

Furthermore, since $i : X \to Y$ is a cofibration between cofibrant objects, the iterated pushout product $i^{\boxprod r} : Q^r_{r-1}(i) \to Y^{\otimes r}$ is also a cofibration between cofibrant objects in $\M$ by the pushout product axiom.  See, for example, \cite{harper-jpaa} (proof of 7.19) for an explicit iterated construction of $Q^r_{r-1}$.  As every $\sO_{A^k}$ is entrywise cofibrant, both objects
\[\sO_{A^{t-1}}\drcub  \otimes Q^r_{r-1}(i) \andspace \sO_{A^{t-1}}\drcub  \otimes Y^{\otimes r}\]
are entrywise cofibrant in $\M$ by the pushout product axiom.  Condition $(\filledstar)$ (Def. \ref{def:star}) in $\M$ implies that, after taking $\Sigma_r$-coinvariants, both objects on the left side of \eqref{star.r.t} are entrywise cofibrant in $\M$.

Now we apply the left Quillen equivalence \cite{hirschhorn} (11.6.5(2))
\[L : \Msigmabopd \to \Nsigmabopd\]
to the countable composition \eqref{osubstar.t} and the pushouts \eqref{star.r.t}.  We obtain the countable composition
\begin{equation}\label{losubstar.t}
\narrowxy{L\sO_{A^{t-1}}\dbrb = L\sO_{A^{t-1}}^0\dbrb \ar[r] & L\sO_{A^{t-1}}^1\dbrb \ar[r] & L\sO_{A^{t-1}}^2\dbrb \ar[r] & \cdots \ar[r] & L\sO_{A^t}\dbrb}
\end{equation}
in $\Nsigmabopd$ of entrywise cofibrations between entrywise cofibrant objects.  For each $r \geq 1$, the $r$th map is the pushout
\begin{equation}\label{lstar.r.t}
\nicexy{L\Bigl[\sO_{A^{t-1}}\drcub \tensoroversigmar Q^r_{r-1}(i)\Bigr] \ar[r] \ar[d]_{L\bigl(\Id \tensoroversigmar i^{\boxprod r}\bigr)} & L\sO_{A^{t-1}}^{r-1}\dbrb \ar[d]\\
L\Bigl[\sO_{A^{t-1}}\drcub \tensoroversigmar Y^{\otimes r}\Bigr] \ar[r] & L\sO^r_{A^{t-1}}\dbrb}
\end{equation}
in $\Nsigmabopd$ with both vertical maps entrywise cofibrations between entrywise cofibrant objects.  

Next, applying Proposition \ref{o-ainfinity} to the pushout \eqref{lbar.at.pushout}, we see that the map
\[\sP_{\Lbar A^{t-1}}\dbrb \to \sP_{\Lbar A^t}\dbrb\]
is a countable composition
\begin{equation}\label{psub.lbarstar.t}
\narrowxy{\sP_{\Lbar A^{t-1}}\dbrb = \sP_{\Lbar A^{t-1}}^0\dbrb \ar[r] & \sP_{\Lbar A^{t-1}}^1\dbrb \ar[r] & \sP_{\Lbar A^{t-1}}^2\dbrb \ar[r] & \cdots \ar[r] & \sP_{\Lbar A^t}\dbrb}
\end{equation}
in $\Nsigmabopd$ in which, for each $r \geq 1$, the $r$th map is the pushout
\begin{equation}\label{lbarstar.r.t}
\nicexy{\sP_{\Lbar A^{t-1}}\drcub \tensoroversigmar Q^r_{r-1}(Li) \ar[r] \ar[d]_{\Id \tensoroversigmar (Li)^{\boxprod r}} & \sP_{\Lbar A^{t-1}}^{r-1}\dbrb \ar[d]\\
\sP_{\Lbar A^{t-1}}\drcub \tensoroversigmar (LY)^{\otimes r} \ar[r] & \sP^r_{\Lbar A^{t-1}}\dbrb}
\end{equation}
in $\Nsigmabopd$.  We now argue as in the two paragraphs before \eqref{losubstar.t} and use conditions $\clubcof$ and $(\filledstar)$ in $\N$.  We then see that in \eqref{psub.lbarstar.t} every map is an entrywise cofibration between entrywise cofibrant objects.  Moreover, both vertical maps in \eqref{lbarstar.r.t} are entrywise cofibrations between entrywise cofibrant objects.

Consider the commutative ladder diagram from \eqref{losubstar.t} to \eqref{psub.lbarstar.t},
\begin{equation}\label{lostart.to.plbarstart}
\narrowxy{L\sO_{A^{t-1}}\dbrb = L\sO_{A^{t-1}}^0\dbrb \ar[d]_-{f^0_{t-1}}^-{\defn f_{t-1}}\ar[r] & L\sO_{A^{t-1}}^1\dbrb \ar[d]_-{f_{t-1}^1} \ar[r] & L\sO_{A^{t-1}}^2\dbrb \ar[d]_-{f_{t-1}^2} \ar[r] & \cdots \ar[r] & L\sO_{A^t}\dbrb \ar[d]_-{\colim_r\, f_{t-1}^r \,=}^-{f_t}\\
 \sP_{\Lbar A^{t-1}}\dbrb = \sP_{\Lbar A^{t-1}}^0\dbrb \ar[r] & \sP_{\Lbar A^{t-1}}^1\dbrb \ar[r] & \sP_{\Lbar A^{t-1}}^2\dbrb \ar[r]  & \cdots \ar[r] & \sP_{\Lbar A^t}\dbrb}
\end{equation}
in $\Nsigmabopd$.  By \cite{hirschhorn} (15.10.12(1)), to show that the colimit $f_t$ is a weak equivalence, it suffices to show that each vertical map $f_{t-1}^r$ for $r \geq 0$ is a weak equivalence.  The initial map $f_{t-1}^0$ is defined as $f_{t-1}$, which is a weak equivalence.

For the induction step, suppose $r \geq 1$ and that $f_{t-1}^{r-1}$ is a weak equivalence.  We want to show that $f_{t-1}^r$ is a weak equivalence. Consider the naturally induced commutative cube from \eqref{lstar.r.t} (the back face below) to \eqref{lbarstar.r.t} (the front face),
\begin{equation}\label{lstarrt.to.lbarstarrt}
\nicexy@C-1.3cm{L\Bigl[\sO_{A^{t-1}}\drcub \tensoroversigmar Q^r_{r-1}(i)\Bigr] \ar[rr] \ar[dd]_{L\bigl(\Id \tensoroversigmar i^{\boxprod r}\bigr)} \ar@(d,dl)[dr]^-{\alpha} && L\sO_{A^{t-1}}^{r-1}\dbrb \ar'[d][dd] \ar[dr]^-{f_{t-1}^{r-1}}_-{\sim} &\\
& \sP_{\Lbar A^{t-1}}\drcub \tensoroversigmar Q^r_{r-1}(Li) \ar[rr] \ar[dd]  && \sP_{\Lbar A^{t-1}}^{r-1}\dbrb \ar[dd]\\
L\Bigl[\sO_{A^{t-1}}\drcub \tensoroversigmar Y^{\otimes r}\Bigr] \ar'[r][rr] \ar@(d,dl)[dr]_-{\beta} && L\sO^r_{A^{t-1}}\dbrb \ar[dr]^-{f_{t-1}^r} &\\
& \sP_{\Lbar A^{t-1}}\drcub \tensoroversigmar (LY)^{\otimes r} \ar[rr] && \sP^r_{\Lbar A^{t-1}}\dbrb}
\end{equation}
in $\Nsigmabopd$.  We will prove in Lemma \ref{top.face.of.cube} that the top and bottom faces of the cube \eqref{lstarrt.to.lbarstarrt} are indeed commutative. The map $\alpha$ factors as the composite
\[\nicexy@C+.7cm{L\Bigl[\sO_{A^{t-1}}\drcub \tensoroversigmar Q^r_{r-1}(i)\Bigr]  \ar[r]^-{\alpha} \ar[d]_-{\cong} & \sP_{\Lbar A^{t-1}}\drcub \tensoroversigmar Q^r_{r-1}(Li) \\
\Bigl[L\bigl(\sO_{A^{t-1}}\drcub \otimes Q^r_{r-1}(i)\bigr)\Bigr]_{\Sigma_r} \ar[d]_-{\alpha_1 \,=\, (L^2)_{\Sigma_r}} & \Bigl[L\sO_{A^{t-1}}\drcub \otimes Q^r_{r-1}(Li)\Bigr]_{\Sigma_r} \ar[u]^-{\alpha_2 \,=}_-{f_{t-1} \tensoroversigmar \Id} \\
[L\sO_{A^{t-1}}\drcub \otimes LQ_{r-1}^r(i)]_{\Sigma_r} \ar[ur]_-{\Phi_r^*}
&}\]
in which:
\begin{itemize}
\item $L^2$ is the comonoidal structure map of $L$ \eqref{comonoidal.map};
\item $\Phi_r: LQ^r_{r-1}(i) \to Q^r_{r-1}(Li)$ is a weak equivalence between cofibrant objects, defined in Lemma \ref{lemma:2018};
\item $\Phi_r^* = [\Id_{L\sO_{A^{t-1}}\drcub} \otimes \; \Phi_r]_{\Sigma_r}$.
\end{itemize}
Observe that to define the map $\alpha_2$, the map $f_{t-1}$ must be an equivariant map, instead of merely a map of individual entries.  There is a similar factorization for the map $\beta$.  In the top face, the top horizontal map is induced by $L$ applied to the map $\sO \comp X \to A^{t-1}$, while the other horizontal map is induced by the map
\[\nicexy{\sP \comp (LX) = \Lbar(\sO \comp X) \ar[r] & \Lbar A^{t-1}.}\]

We already observed above that both vertical maps in the left face of (\ref{lstarrt.to.lbarstarrt}) are entrywise cofibrations and that all the objects in the cube are entrywise cofibrant.  By the Cube Lemma \cite{hovey} (5.2.6), to show that $f_{t-1}^r$ is a weak equivalence, it suffices to show that both $\alpha$ and $\beta$ are entrywise weak equivalences.
 
To show that $\alpha$ is a weak equivalence, it is enough to show that $\alpha_1$, $\alpha_2$, and $\Phi_r^*$ are weak equivalences.  We already observed that $\sO_{A^{t-1}}$ is entrywise cofibrant and that  $Q^r_{r-1}(i)$ is cofibrant in $\M$.  So $\alpha_1$ is a weak equivalence by condition $(\#)$ (Def. \ref{def:sharp}). Since $L$ preserves cofibrancy, condition $(\medstar)$ in $\N$ (Def. \ref{def:star}), together with Lemma \ref{lemma:2018}, implies $\Phi^*$ is a weak equivalence. Likewise, since $f_{t-1}$ is a weak equivalence between entrywise cofibrant objects, condition $(\medstar)$ in $\N$ implies that $\alpha_2$ is a weak equivalence.  This proves that $\alpha$ is a weak equivalence.  A similar argument, with $Y^{\otimes r}$ in place of $Q^r_{r-1}(i)$, proves that $\beta$ is a weak equivalence.

Therefore, the map $f^r_{t-1}$ is a weak equivalence.  This finishes the induction in the ladder diagram \eqref{lostart.to.plbarstart}, proving that the map $f_t$ is a weak equivalence.  This in turn proves the induction step in the first ladder diagram \eqref{lostar.to.plbarstar}, so the map $f_\infty$ is a weak equivalence.

Finally, the assertion $f_\infty\dnothing = (\chi_A)_d$ is a consequence of the naturality of $f_\infty$ and Example \ref{comparison.initial}.
\end{proof}

\begin{lemma}
\label{top.face.of.cube}
The top and bottom faces of the cube \eqref{lstarrt.to.lbarstarrt} are commutative.
\end{lemma}

\begin{proof}
The top face of the cube  \eqref{lstarrt.to.lbarstarrt} is the diagram
\begin{equation}\label{top.face}
\nicexy{L\Bigl[\sO_{A^{t-1}}\drcub \tensoroversigmar Q^r_{r-1}(i)\Bigr] \ar[r]^-{g_*} \ar[d]_-{\alpha} & L\sO_{A^{t-1}}^{r-1}\dbrb \ar[d]^-{f_{t-1}^{r-1}}\\
\sP_{\Lbar A^{t-1}}\drcub \tensoroversigmar Q^r_{r-1}(Li) \ar[r]^-{g'_*} & \sP_{\Lbar A^{t-1}}^{r-1}\dbrb}
\end{equation}
with $r,t \geq 1$ and $i : X \to Y \in \M$ concentrated in a single color $c \in \fC$.  The top horizontal map is induced by the $\sO$-algebra map $g : \sO \comp X \to A^{t-1}$, whose adjoint $X \to A^{t-1} \in \mtoc$ is also denoted by $g$.  The bottom horizontal map is induced by the composite
\[\nicexy{LX \ar@/^2pc/@(ul,ur)[rr]^-{g'} \ar[r]^-{Lg} & LA^{t-1} \ar[r]^-{\chi_A} & \Lbar A^{t-1},}\]
which is adjoint to the map 
\[\nicexy{\sP \comp (LX) = \Lbar(\sO \comp X) \ar[r]^-{\Lbar g} & \Lbar A^{t-1}}\]
of $\sP$-algebras.  As before, we omit writing the forgetful functors.

To show that \eqref{top.face} is commutative, observe that $\sO_{A^{t-1}}$ and $\sP_{\Lbar A^{t-1}}$ are defined as coequalizers (Def. \ref{oaalgebra}), that $Q^r_{r-1}$ is a colimit indexed by the punctured $r$-cube $\{0 < 1\}^r \setminus \{(1,\ldots,1)\}$ \cite{harper-jpaa} (7.19), and that taking $\Sigma_r$-coinvariants is also a colimit.  Therefore, it is enough to check the commutativity of the diagram \eqref{top.face} when it is restricted to a typical node in the colimiting cone.  In other words, it is enough to check the commutativity of the solid-arrow diagram
\begin{equation}\label{top.face.restricted}
\begin{footnotesize}
\nicexy@R-.3cm@C-1.3cm{L\Bigl[\sO\duarcub \otimes A^{t-1}_{[\ua]} \otimes X^{p} \otimes Y^q\Bigr] \ar@(r,l)[dr]|-{g_*} \ar[dd]|-{\mathrm{comonoidal}} && L\Bigl[\sO_{A^{t-1}}\dqcub \otimes Y^q \Bigr] \ar[dd]|-{\mathrm{comonoidal}}\\
&L\Bigl[\sO(\vdots) \otimes A^{t-1}_{[\ua,pc]} \otimes Y^q\Bigr] \ar@(r,l)[ur]|-{\mathrm{natural}} \ar@{.>}[d] &\\ 
L\sO(\vdots) \otimes (LA^{t-1})_{[\ua]} \otimes (LX)^p \otimes (LY)^q \ar[d]_-{(\chi_{A^{t-1}})_*} & \bullet \ar@{.>}[d] & L\sO_{A^{t-1}}\dqcub \otimes (LY)^q \ar[ddd]|-{(f_{t-1}, \Id)}\\
L\sO(\vdots) \otimes (\Lbar A^{t-1})_{[\ua]} \otimes (LX)^p \otimes (LY)^q \ar[dd]|-{(\fbar,\Id)} & \bullet \ar@{.>}[d] &\\
& \sP(\vdots) \otimes (\Lbar A^{t-1})_{[\ua,pc]} \otimes (LY)^q \ar@(r,l)[dr]|-{\mathrm{natural}}&\\
\sP(\vdots) \otimes (\Lbar A^{t-1})_{[\ua]} \otimes (LX)^p \otimes (LY)^q \ar@(r,l)[ur]|-{g'_*} && \sP_{\Lbar A^{t-1}}\dqcub \otimes (LY)^q}
\end{footnotesize}
\end{equation}
in which $\ua \in \pofc$ is arbitrary, $(\vdots) = \duarcub$, and $p+q = r$ with $p>0$ (hence $0 \leq q < r$).  We will show that this diagram is commutative by factoring it into two commutative diagrams as indicated by the dotted arrows.

The left half of the diagram \eqref{top.face.restricted} is the commutative diagram
\begin{equation}\label{top.face.left}
\begin{small}
\nicexy@C-.3cm{L\Bigl[\sO\duarcub \otimes A^{t-1}_{[\ua]} \otimes X^{p} \otimes Y^q\Bigr] \ar[r]^-{g_*} \ar[d]|-{\mathrm{comonoidal}} & L\Bigl[\sO(\vdots) \otimes A^{t-1}_{[\ua,pc]} \otimes Y^q\Bigr] \ar[d]|-{\mathrm{comonoidal}}\\ 
L\sO(\vdots) \otimes (LA^{t-1})_{[\ua]} \otimes (LX)^p \otimes (LY)^q \ar[d]_-{(\chi_{A^{t-1}})_*} \ar[r]^-{g_*} & L\sO(\vdots) \otimes (LA^{t-1})_{[\ua,pc]} \otimes (LY)^q  \ar[d]^-{(\chi_{A^{t-1}})_*}\\
L\sO(\vdots) \otimes (\Lbar A^{t-1})_{[\ua]} \otimes (LX)^p \otimes (LY)^q \ar[d]_-{(\fbar,\Id)} \ar[r]^-{g'_*} & L\sO(\vdots) \otimes (\Lbar A^{t-1})_{[\ua,pc]} \otimes (LY)^q  \ar[d]^-{(\fbar,\Id)} &\\
\sP(\vdots) \otimes (\Lbar A^{t-1})_{[\ua]} \otimes (LX)^p \otimes (LY)^q \ar[r]^-{g'_*}
& \sP(\vdots) \otimes (\Lbar A^{t-1})_{[\ua,pc]} \otimes (LY)^q}
\end{small}
\end{equation}
in which the top square is commutative by naturality.  The middle square is commutative by the definition of the map $g'$, and the bottom square is commutative by definition.

It remains to show that the right half of the diagram \eqref{top.face.restricted} is commutative.  First observe that the upper right region of the diagram \eqref{top.face.restricted} can be rewritten as in the commutative diagram:
\begin{equation}\label{upper.right}
\nicexy{L\Bigl[\sO(\vdots) \otimes A^{t-1}_{[\ua,pc]} \otimes Y^q\Bigr] \ar[r]^-{\mathrm{natural}} \ar[d]|-{\mathrm{comonoidal}} & L\Bigl[\sO_{A^{t-1}}\dqcub \otimes Y^q \Bigr]  \ar[d]|-{\mathrm{comonoidal}}\\ 
L\Bigl[\sO(\vdots) \otimes A^{t-1}_{[\ua,pc]}\Bigr] \otimes (LY)^q \ar[r]^-{\mathrm{natural}} &  L\sO_{A^{t-1}}\dqcub \otimes (LY)^q}
\end{equation}
Therefore, to show that the right half of the diagram \eqref{top.face.restricted} is commutative, it is enough to show that the diagram
\begin{equation}\label{right.half.1}
\begin{footnotesize}
\nicexy@R-.3cm@C-1.3cm{L\Bigl[\sO(\vdots) \otimes A^{t-1}_{[\ua,pc]} \otimes Y^q\Bigr] \ar[dr]|-{\mathrm{comonoidal}} \ar[dd]|-{\mathrm{comonoidal}} && L\sO_{A^{t-1}}\dqcub \otimes (LY)^q \ar[dddddd]|-{(f_{t-1},\Id)}\\
& L\Bigl[\sO(\vdots) \otimes A^{t-1}_{[\ua,pc]}\Bigr] \otimes (LY)^q \ar[ur]|-{\mathrm{natural}} \ar[dl]|-{\mathrm{comonoidal}} \ar[dd]|-{L(f,\mathrm{unit})} &\\
L\sO(\vdots) \otimes (LA^{t-1})_{[\ua,pc]} \otimes (LY)^q \ar[dd]_-{(\chi_{A^{t-1}})_*} \ar[dddr]|-{\bigl(Lf,L(\mathrm{unit})\bigr)} &&\\  
& L\Bigl[R\sP(\vdots) \otimes (R\Lbar A^{t-1})_{[\ua,pc]}\Bigr] \otimes (LY)^q \ar[dd]|-{\mathrm{comonoidal}} &\\
L\sO(\vdots) \otimes (\Lbar A^{t-1})_{[\ua,pc]} \otimes (LY)^q  \ar[dd]|-{(\fbar,\Id)}&&\\
&  LR\sP(\vdots) \otimes (LR\Lbar A^{t-1})_{[\ua,pc]} \otimes (LY)^q \ar[dl]|-{\mathrm{counit}}&\\
\sP(\vdots) \otimes (\Lbar A^{t-1})_{[\ua,pc]} \otimes (LY)^q \ar[rr]^-{\mathrm{natural}} && \sP_{\Lbar A^{t-1}}\dqcub \otimes (LY)^q}
\end{footnotesize}
\end{equation}
is commutative.  The left column in \eqref{right.half.1} is equal to the right column in \eqref{top.face.left} (i.e. the middle column in \ref{top.face.restricted}). The path from the top left to the lower right, across the top then down the right side, is isomorphic to the right column in \ref{top.face.restricted}. The right sub-diagram is commutative by the definition of $f_{t-1}$.  The top left triangle (with three comonoidal maps) and the middle left triangle (with two comonoidal maps) are commutative by naturality.  The lower left triangle is commutative because, by adjunction, $\fbar : L\sO \to \sP$ is equal to the composite
\[\nicexy{L\sO \ar[r]^-{Lf} & LR\sP \ar[r]^-{\mathrm{counit}} & \sP.}\]
Likewise, the comparison map $\chi_{A^{t-1}} : LA^{t-1} \to \Lbar A^{t-1}$ is equal to the composite
\[\nicexy{LA^{t-1} \ar[r]^-{L(\mathrm{unit})} & LR\Lbar A^{t-1} \ar[r]^-{\mathrm{counit}} & \Lbar A^{t-1}.}\]
Therefore, the diagram \eqref{right.half.1} is commutative.  As discussed above, together with the commutative diagrams \eqref{top.face.left} and \eqref{upper.right}, we conclude that the diagrams \eqref{top.face.restricted} and, therefore, \eqref{top.face} are commutative. The proof for the bottom case is similar, but is much simpler (corresponding to $p=0$ in the proof above).
\end{proof}

\begin{lemma} \label{lemma:2018}
If $i : X \to Y$ is a cofibration between cofibrant objects in $\calm$, and $L : \calm \to \caln$ is as in the statement of Proposition \ref{loa.to.pla}, then the natural map 

\[\Phi_r : L Q^r_{r-1}(i) \to Q^r_{r-1}(Li)\]

is a weak equivalence between cofibrant objects in $\caln$.
\end{lemma}

The map $\Phi_r$ is defined explicitly in the proof below.

\begin{proof}
First, the domain and codomain of $\Phi_r$ are cofibrant because $L$ preserves cofibrations and cofibrant objects, and $Q_{r-1}^r(g)$ is cofibrant whenever $g$ is a cofibration between cofibrant objects. To prove $\Phi_r$ is a weak equivalence, we proceed by induction. For the base case, $r=1$, $\Phi_r$ is an isomorphism, because $Q_0^1(i) = X$, so $LQ_0^1(i) = LX \cong Q_0^1(Li:LX\to LY)$. For the inductive step, we recall $Q_{r-1}^r(i)$ may be inductively constructed (\cite{harper-jpaa} (proof of 7.19)) as a pushout in $\calm$:

\begin{equation*}
\nicexy{
Q_{r-2}^{r-1}(i) \otimes X \ar[r]^-{\Id \otimes i} \ar[d]_-{i^{\boxprod r-1} \otimes \Id_X} & Q_{r-2}^{r-1}(i) \otimes Y \ar[d]\\ 
Y^{\otimes r-1} \otimes X \ar[r] & Q_{r-1}^r(i)}
\end{equation*}

We may construct $Q_{r-1}^r(Li)$ analogously, as a pushout in $\N$. When we apply $L$ to the pushout square above, we achieve the following commutative cube in $\N$, which defines $\Phi_r$, and where the front and back faces are pushout squares (because $L$ is a left adjoint):

\begin{equation} \label{cube:2018}
\nicexy{
L(Q_{r-2}^{r-1}(i) \otimes X) \ar[rr]^-{L(\Id_{Q(i)} \otimes i)} \ar[dd]_-{L(i^{\boxprod r-1} \otimes \Id_X)} \ar[dr]^{\delta_2} && L(Q_{r-2}^{r-1}(i) \otimes Y) \ar'[d][dd] \ar[dr]^{\delta_3}\\
&  Q_{r-2}^{r-1}(Li) \otimes LX \ar[rr]^(0.3){\Id_{Q(Li)} \otimes Li} \ar[dd]_(0.3){(Li)^{\boxprod r-1} \otimes \Id_{LX}} && Q_{r-2}^{r-1}(Li) \otimes LY \ar[dd]\\ 
L(Y^{\otimes r-1} \otimes X) \ar'[r][rr] \ar[dr]_{\delta_1} && L(Q_{r-1}^r(i)) \ar[dr]^{\Phi_r} \\
& (LY)^{\otimes r-1} \otimes LX \ar[rr] && Q_{r-1}^r(Li)}
\end{equation}

In the top face, $\Id \otimes i$ is a cofibration between cofibrant objects, hence so is $L(\Id_{Q(i)} \otimes i)$. Since $Li$ is a cofibration between cofibrant objects, $Q_{r-2}^{r-1}(Li)$ is cofibrant and in the front face, $\Id_{Q(Li)}\otimes Li$ is a cofibration between cofibrant objects. Hence, all eight objects in (\ref{cube:2018}) are cofibrant, and by the Cube Lemma \cite{hovey} (5.2.6), to prove $\Phi_r$ is a weak equivalence, it suffices to prove $\delta_1, \delta_2,$ and $\delta_3$ are weak equivalences. 

The map $\delta_1$ is the iteration of the comonoidal structure map $L^2$. Since $X$ and $Y$ are cofibrant, $\delta_1$ is a composite of weak equivalences, by Definition \ref{def:weak.symmetric.monoidal}(1). The map $\delta_2$ factors as follows:

\begin{equation*}
\nicexy{
L(Q_{r-2}^{r-1}(i) \otimes X) \ar[rr]^{\delta_2} \ar[dr]_{L^2_{Q(i),X}}&& Q_{r-2}^{r-1}(Li) \otimes LX  \\
& L(Q_{r-2}^{r-1}(i)) \otimes LX \ar[ur]_{\Phi_{r-1} \otimes \Id_{LX}} & 
}
\end{equation*}

The comonoidal structure map $L^2_{Q(i),X}$ is a weak equivalence by Definition \ref{def:weak.symmetric.monoidal}(1). The map $\Phi_{r-1}$ is a weak equivalence between cofibrant objects by the inductive hypothesis. Since $LX$ is cofibrant, $-\otimes LX$ is a left Quillen functor and hence preserves weak equivalences between cofibrant objects, by Ken Brown's Lemma \cite{hovey} (1.1.12). It follows that $\delta_2$ is a weak equivalence. Similarly, $\delta_3$ factors as:

\begin{equation*}
\nicexy{
L(Q_{r-2}^{r-1}(i) \otimes Y) \ar[rr]^{\delta_3} \ar[dr]_{L^2_{Q(i),Y}}&& Q_{r-2}^{r-1}(Li) \otimes LY  \\
& L(Q_{r-2}^{r-1}(i)) \otimes LY \ar[ur]_{\Phi_{r-1} \otimes \Id_{LY}} & 
}
\end{equation*}

The same argument as we used for $\delta_2$ proves that $\delta_3$ is a weak equivalence. This completes our inductive proof that $\Phi_r$ is a weak equivalence between cofibrant objects for all $r > 0$.
\end{proof}

\subsection{Main Theorem}

The following theorem is our main result.  Roughly speaking, it says that operadic algebras are homotopically well-behaved with respect to Quillen equivalences.

\begin{theorem}[Lifting Quillen Equivalences]
\label{main.theorem}
Suppose:
\begin{enumerate}
\item $L : \M \adjoint \N : R$ is a nice Quillen equivalence (Def. \ref{def:nice.qeq}).
\item $f : \sO \to R\sP$ is a map of $\fC$-colored operads in $\M$ with $\fC$ a set, $\sO$ an entrywise cofibrant $\fC$-colored operad in $\M$, and $\sP$ an entrywise cofibrant $\fC$-colored operad in $\N$.  The entrywise adjoint $\fbar : L\sO \to \sP$ is an entrywise weak equivalence in $\N$. 
\end{enumerate}
Then the lifted adjunction \eqref{lbar.ocomp.diagram}
\[\nicexy{\algo \ar@<2.5pt>[r]^-{\Lbar} & \algp \ar@<2.5pt>[l]^-{R}}\]
is a Quillen equivalence between the semi-model categories of $\sO$-algebras in $\M$ and of $\sP$-algebras in $\N$ (Theorem \ref{theorem623}).
\end{theorem}

Below, we will show that, if $\sO$ and $\sP$ are $\Sigma$-cofibrant, then we can weaken assumption (1) above to only requiring that $(L,R)$ be a weak symmetric monoidal Quillen equivalence and that the domains of the generating cofibrations in $\M$ be cofibrant.

\begin{proof}
Recall that weak equivalences and fibrations in $\algo$ and $\algp$ are defined entrywise in $\M$ and $\N$, respectively.  The lifted adjunction $\Lbar \dashv R$ is a Quillen adjunction--i.e., the right adjoint $R$ preserves fibrations and trivial fibrations--because $UR = RU$ in the diagram \eqref{lbar.ocomp.diagram}.

To see that $\Lbar \dashv R$ is a Quillen equivalence between semi-model categories, suppose $A$ is a cofibrant $\sO$-algebra, $B$ is a fibrant $\sP$-algebra, and $\varphi : \Lbar A \to B \in \algp$.  We want to show that $\varphi$ is a weak equivalence if and only if its adjoint $\varphibar : A \to RB$ is a weak equivalence.  By Proposition \ref{loa.to.pla} the comparison map $\chi_A : LA \to \Lbar A \in \ntoc$ is an entrywise weak equivalence.  By the $2$-out-of-$3$ property, $\varphi$ is a weak equivalence if and only if the composite
\[\nicexy{LA \ar[r]^-{\chi_A}_-{\sim} & \Lbar A \ar[r]^-{\varphi} & B \in \ntoc}\]
is a weak equivalence.  Note that $B \in \ntoc$ is fibrant and that $A \in \mtoc$ is cofibrant by Theorem \ref{theorem623}(2).  Since the entrywise prolongation
\[L : \mtoc \adjoint \ntoc : R\]
is a Quillen equivalence \cite{hirschhorn} (11.6.5(2)), the map $\varphi\chi_A \in \ntoc$ is a weak equivalence if and only if its adjoint $\varphibar : A \to RB \in \mtoc$ is a weak equivalence (Example \ref{comparison.adjoint}).  This proves that $\Lbar \dashv R$ is a Quillen equivalence.
\end{proof}

Note that in Theorem \ref{main.theorem}, we only ask that the colored operads $\sO$ and $\sP$ be \emph{entrywise} cofibrant, instead of the much stronger conditions of being $\Sigma$-cofibrant and admissible \cite{bm03} (section 4). In particular, our operads will almost never be admissible. 

\subsection{$\Sigma$-cofibrant colored operads}

A colored operad $\sO$ is $\Sigma$-cofibrant if it is cofibrant in the product model structure $\symseqc$, where each category $\calm^{\Sigma_{\fC}^{op} \times \{d\}}$ is given the projective model structure \cite{hirschhorn} (11.6.1). If we require $\sO$ and $\sP$ to be $\Sigma$-cofibrant, then we can weaken our conditions on the adjunction $(L,R)$. The results that follow provide an extension of Proposition 12.3.4 in \cite{fresse-book}, which considers one operad acting in two different model categories.

\begin{proposition}
\label{loa.to.pla.Sigma.cof}
Suppose:
\begin{enumerate}
\item $L : \M \adjoint \N : R$ is a weak symmetric monoidal Quillen equivalence (Def. \ref{def:weak.symmetric.monoidal}).
\item $f : \sO \to R\sP$ is a map of $\fC$-colored operads in $\M$ with $\fC$ a set, $\sO$ a $\Sigma$-cofibrant $\fC$-colored operad in $\M$, and $\sP$ a $\Sigma$-cofibrant $\fC$-colored operad in $\N$.  The entrywise adjoint $\fbar : L\sO \to \sP$ is an entrywise weak equivalence in $\N$. 
\item Every generating cofibration in $\M$ has cofibrant domain.
\end{enumerate}
Suppose $A$ is a cofibrant $\sO$-algebra.  Then the map $f : \sO \to R\sP$ induces a natural entrywise weak equivalence
\[\nicexy{L\sO_A \ar[r]^-{f_{\infty}} & \sP_{\Lbar A} \in \symseqcn}\]
whose value at $\dnothing$ is the comparison map $\chi_A : LA \to \Lbar A$ \eqref{comparison.map} evaluated at $d$ for each $d \in \fC$.
\end{proposition}

\begin{proof}
The proof proceeds exactly as in Proposition \ref{loa.to.pla}. Instead of Lemma \ref{middle-row-lemma}, we use the colored version of Proposition 5.17 in \cite{harper-gnt}, which implies $\sO_{A^t}$ and $\sP_{\overline{L}A^t}$ are $\Sigma$-cofibrant for all $t$.
Instead of $\clubcof$ in \ref{star.r.t}, we use the observation that, for a $\Sigma_r$-projectively cofibrant object $X$, the functor $X \otimes_{\Sigma_r} -$ is left Quillen. In the proof, $X$ is first $\sO_{A^{t-1}}\drcub$ in \eqref{star.r.t} and is later $\sP_{\Lbar A^{t-1}}\drcub$ in \eqref{lbarstar.r.t}. Rather than condition $(\filledstar)$, observe that $X \otimes i^{\boxprod r}$ is a cofibration in the projective model structure on $\M^{\Sigma_r}$, and the domain and codomain are projectively cofibrant.  Thus, after taking $\Sigma_r$-coinvariants, we are left with a cofibration between cofibrant objects in $\M$. Similarly, $X\otimes_{\Sigma_r} (Li)^{\boxprod r}$ is a cofibration between cofibrant objects in $\N$.
 
Finally, when proving the maps $\alpha$ and $\beta$ are weak equivalences, instead of using conditions $(\#)$ and $(\medstar)$, we use that $(L,R)$ is a weak monoidal Quillen pair.
This implies $L^2$ is a weak equivalence between $\Sigma_r$-projectively cofibrant objects (since $L$ induces a left Quillen functor on $\M^{\Sigma_r}$). It follows from Ken Brown's Lemma that $\alpha_1$ is a weak equivalence. The situation for $\alpha_2$ is similar, since $f_{t-1}\otimes \Id$ is a weak equivalence between $\Sigma_r$-projectively cofibrant objects. It follows that $\alpha$ is a weak equivalence, a similar argument shows that $\beta$ is a weak equivalence, and then the double induction demonstrates that $f_\infty$ is a weak equivalence as required.
\end{proof}

Similarly, we have a version of Theorem \ref{main.theorem} for $\Sigma$-cofibrant colored operads:

\begin{theorem}[Lifting Quillen Equivalences for $\Sigma$-Cofibrant Operads]
\label{main.theorem.Sigma}
Suppose:
\begin{enumerate}
\item $L : \M \adjoint \N : R$ is a weak symmetric monoidal Quillen equivalence (Def. \ref{def:weak.symmetric.monoidal}).
\item $f : \sO \to R\sP$ is a map of $\fC$-colored operads in $\M$ with $\fC$ a set, $\sO$ a $\Sigma$-cofibrant $\fC$-colored operad in $\M$, and $\sP$ a $\Sigma$-cofibrant $\fC$-colored operad in $\N$.  The entrywise adjoint $\fbar : L\sO \to \sP$ is an entrywise weak equivalence in $\N$. 
\item Every generating cofibration in $\M$ has a cofibrant domain.
\end{enumerate}
Then the lifted adjunction \eqref{lbar.ocomp.diagram}
\[\nicexy{\algo \ar@<2.5pt>[r]^-{\Lbar} & \algp \ar@<2.5pt>[l]^-{R}}\]
is a Quillen equivalence between the semi-model categories of $\sO$-algebras in $\M$ and of $\sP$-algebras in $\N$.
\end{theorem}

\begin{proof}
The proof proceeds exactly as in Theorem \ref{main.theorem}, but uses Proposition \ref{loa.to.pla.Sigma.cof} instead of Proposition \ref{loa.to.pla}. The existence of the semi-model structures on $\algo$ and $\algp$ is now due to Theorem 6.3.1 in \cite{white-yau}, which also proves that cofibrant $\sO$-algebras are cofibrant in $\mtoc$, avoiding the need for Theorem \ref{theorem623} and $(\clubsuit)$. 
\end{proof}

\section{Special Cases: Rectification and Derived Change of Category} \label{sec:rect-and-change}

In this section, we discuss special cases of our main results, Theorems \ref{main.theorem} and \ref{main.theorem.Sigma}. We begin with the strongest possible condition on $L$ (that it is the identity), and successively weaken our conditions on $L$. We see that rectification, change of rings, and change of underlying model category (i.e., lifting Quillen equivalences) are all special cases of the same general framework.

\subsection{Rectification}

Restricting Theorem \ref{main.theorem} to the special case $L = R = \Id$ (so condition $(\#)$ (Def. \ref{def:sharp}) holds automatically), we obtain the following rectification result for entrywise cofibrant operads.

\begin{corollary}[Rectification of Operadic Algebras]
\label{cor.rectification}
Suppose $\M$ is a cofibrantly generated monoidal model category satisfying the conditions $(\medstar)$, $(\filledstar)$ (Def. \ref{def:star}), and $(\clubsuit)$ (Def. \ref{def:club}),  in which every generating cofibration has a cofibrant domain.  Suppose $\fC$ is a set, and $f : \sO \to \sP$ is a map of entrywise cofibrant $\fC$-colored operads that is an entrywise weak equivalence in $\M$.  Then the induced adjunction
\[\nicexy{\algo \ar@<2.5pt>[r]^-{f_!} & \algp \ar@<2.5pt>[l]^-{f^*}}\]
is a Quillen equivalence between semi-model categories.
\end{corollary}

The right adjoint $f^* : \algp \to \algo$ is given by restriction along the map $f$.  The left adjoint $f_!$ may be constructed as a certain coequalizer \cite{bm07} (section 4).

\begin{example}
Corollary \ref{cor.rectification} applies when $\M = \Chkplus$, the category of non-negatively graded chain complexes of $k$-modules for a characteristic $0$ field $k$, and $f : \sO \to \sP$ is the cofibrant replacement:
\begin{enumerate}
\item  $\sA_{\infty} \to \As$, where $\sA_{\infty}$ is the operad for $A_{\infty}$-algebras \cite{stasheff} and $\As$ is the operad for differential graded algebras.
\item $\sE_{\infty} \to \Com$, where $\sE_{\infty}$ is an $E_{\infty}$ operad \cite{may72} and $\Com$ is the operad for commutative dg algebras.
\item $\sL_{\infty} \to \Lie$, where $\sL_{\infty}$ is the operad for $L_{\infty}$-algebras and $\Lie$ is the operad for dg Lie algebras \cite{fmy,lada-markl}.
\end{enumerate}
\end{example}

\begin{remark}
Similar rectification results for \emph{admissible $\Sigma$-cofibrant} operads--as opposed to entrywise cofibrant operads--have been obtained by Berger and Moerdijk \cite{bm03,bm07}.  Admissibility means that the category of algebras over the operad in question has a model category structure in which the fibrations and weak equivalences are defined entrywise in the underlying category. Another rectification result for admissible operads is in \cite{dmitri}. Concrete examples (e.g. \cite{batanin-white-eilenberg}) demonstrate that admissibility is a strong condition, and one should instead expect only semi-model structures.  A rectification result for $1$-colored, entrywise cofibrant, non-symmetric operads is \cite{muro11} (1.3).  For operads in symmetric spectra, a rectification result is \cite{em06} (1.4); a $1$-colored version is \cite{agt1} (1.4).
\end{remark}

\begin{example} \label{ex:change-of-rings}
Rectification results specialize to change of rings results. Let $\M$ be a cofibrantly generated  monoidal model category whose generating cofibrations have cofibrant domains.  Let $f:R\to T$ be a weak equivalence of monoids in $\M$ with $R$ and $T$ cofibrant as objects in $\M$.
\begin{enumerate}
\item Then $f$ induces a Quillen equivalence between the semi-model categories of $R$-modules and $T$-modules. 
\item If $f$ is a map of commutative monoids, then it induces a Quillen equivalence between the semi-model categories of $R$-algebras and $T$-algebras.
\item Suppose $f$ is a map of commutative monoids and $\M$ satisfies the conditions in Corollary \ref{cor.rectification}.  Then $f$ induces a Quillen equivalence between the semi-model categories of commutative $R$-algebras and commutative $T$-algebras.
\end{enumerate}
The first two of these examples date back to \cite{ss}, and can be viewed as special cases of rectification for $\Sigma$-cofibrant operads.   As discussed in Theorem \ref{main.theorem.Sigma}, when the operads involved are $\Sigma$-cofibrant, the conditions $(\medstar)$, $(\filledstar)$ (Def. \ref{def:star}), and $(\clubsuit)$ (Def. \ref{def:club}) are not needed.  The last example for commutative algebras requires the theory of entrywise cofibrant operads as in Theorem \ref{main.theorem}.
\end{example}

\subsection{Modules}

Each (commutative) monoid $T$ in $\M$ admits a category $\Mod(T)$ of left $T$-modules \cite{maclane} (VII.4).  If $L : \M \to \N$ is a lax (symmetric) monoidal functor, then $LT$ is a (commutative) monoid in $\N$, so it admits a category $\Mod(LT)$ of left $LT$-modules.  Using the respective operads for left $T$-modules and left $LT$-modules, Theorem \ref{main.theorem.Sigma} yields the following result.  It has both a commutative version and an associative version, the latter of which is closely related to \cite{ss03} (3.12(1)).

\begin{corollary}[Modules]
\label{cor.com.module}
Suppose  $L : \M \adjoint \N : R$ is a weak (symmetric) monoidal Quillen equivalence with $L$ lax (symmetric) monoidal, and $T$ is a (commutative) monoid that is cofibrant as an object in $\M$. Assume that the domains of the generating cofibrations in $\M$ are cofibrant. Then there is an induced Quillen equivalence
\[\nicexy{\Mod(T) \ar@<2.5pt>[r]^-{\Lbar} & \Mod(LT) \ar@<2.5pt>[l]^-{R}}\]
between the semi-model categories of left $T$-modules in $\M$ and of left $LT$-modules in $\N$.
\end{corollary}

\begin{proof}
The proof is the same in the symmetric and non-symmetric contexts. We need to check condition (2) in Theorem \ref{main.theorem.Sigma}. The $1$-colored operad $\sO$ for left $T$-modules has
\[\sO(n) = \begin{cases} T & \text{ if $n=1$};\\ \varnothing & \text{ if $n\not= 1$.}\end{cases}\]
Likewise, the only non-$\varnothing$ entry in the operad for $LT$-modules is $\sP(1) = LT$.  Both $\sO$ and $\sP$ are $\Sigma$-cofibrant.  The map $f : \sO \to R\sP$ is determined by the unit of the adjunction $T \to RLT$, and $\fbar : L\sO \to \sP$ is the identity map.  
\end{proof}

\subsection{Monoids and Algebras}

Consider Theorem \ref{main.theorem} restricted to the special case with $\sO$ the operad for associative monoids \cite{maclane}(VII.3) in $\M$ and $\sP$ the operad for associative monoids in $\N$.  In this setting, we recover the following result from \cite{ss03} (3.12(3)) with slightly different assumptions.  The slight difference in assumptions is due to the generality of our result. 

\begin{corollary}[Monoids]
\label{cor.schwede.shipley}
Suppose  $L : \M \adjoint \N : R$ is a weak monoidal Quillen equivalence, that the tensor units in $\M$ and $\N$ are cofibrant, and that the domains of the generating cofibrations in $\M$ are cofibrant.  Then there is an induced Quillen equivalence
\[\nicexy{\Monoid(\M) \ar@<2.5pt>[r]^-{\Lbar} & \Monoid(\N) \ar@<2.5pt>[l]^-{R}}\]
between the semi-model categories of monoids in $\M$ and in $\N$.
\end{corollary}

\begin{proof}
As above, we need to check condition (2) in Theorem \ref{main.theorem.Sigma}.  The $1$-colored operad $\sO$ for monoids in $\M$ has
\[\sO(n) = \coprodover{\Sigma_n}\, \tensorunitm\] 
and similarly the $1$-colored operad $\sP$ for monoids in $\N$ has $\sP(n) = \coprod_{\Sigma_n} \tensorunitn$.  In fact, $\sO$ is the image of the associative operad in the category of sets under the strong symmetric monoidal functor
\[\nicexy{\set \ar[r] & \M,}\quad \nicexy{S \ar@{|->}[r] & \coprodover{S}\, \tensorunitm}\]
and similarly for $\sP$.  Both $\sO$ and $\sP$ are $\Sigma$-cofibrant.  The map $\Rbar^0 : L\tensorunitm \to \tensorunitn$ \eqref{unit.adjoint} is a weak equivalence between cofibrant objects in $\N$.  So the coproduct map
\[\nicexy{L\sO(n) = L\Bigl(\coprodover{\Sigma_n}\, \tensorunitm\Bigr) \cong \coprodover{\Sigma_n}\, L\tensorunitm \ar[r]^-{\coprod \Rbar^0} & \coprodover{\Sigma_n}\, \tensorunitn = \sP(n)}\]
is also a weak equivalence.
\end{proof}

A commutative monoid $T$ also admits categories $\alg(T)$ (and $\calg(T)$) of (commutative) $T$-algebras, which are (commutative) monoids in the category of $T$-modules \cite{ss} (p.499). An analogous proof to Corollary \ref{cor.schwede.shipley} demonstrates:

\begin{corollary}[Algebras]
\label{cor.algebras}
Suppose  $L : \M \adjoint \N : R$ is a weak monoidal Quillen equivalence with $L$ lax monoidal, and $T$ is a commutative monoid that is cofibrant as an object in $\M$.  Then there is an induced Quillen equivalence
\[\nicexy{\Alg(T) \ar@<2.5pt>[r]^-{\Lbar} & \Alg(LT) \ar@<2.5pt>[l]^-{R}}\]
between the semi-model categories of $T$-algebras in $\M$ and of $LT$-algebras in $\N$.
\end{corollary}

\begin{proof}
As above, we check condition (2) in Theorem \ref{main.theorem.Sigma}. The $1$-colored operad for $T$-algebras is the enveloping operad $\sO_T$, where $\sO$ is the operad for monoids. In this operad, $T$ replaces the unit $\tensorunitm$. Since $T$ is cofibrant and $\sO$ is $\Sigma$-cofibrant, $\sO_T$ is $\Sigma$-cofibrant by Proposition 5.17 in \cite{harper-gnt}. Similarly, $\sO_{LT}$ is $\Sigma$-cofibrant and $LT$ replaces $\tensorunitn$. The proof now proceeds precisely as above.
\end{proof}

\subsection{Commutative Monoids, Commutative Algebras, and Non-Symmetric Operads}
The following two special cases of Theorem \ref{main.theorem} are the commutative versions of the previous two results.

\begin{corollary}[Commutative Monoids]
\label{cor.com.monoids}
Suppose  $L : \M \adjoint \N : R$ is a nice Quillen equivalence (Def. \ref{def:nice.qeq}) and that the tensor units in $\M$ and $\N$ are cofibrant.  Then there is an induced Quillen equivalence
\[\nicexy{\CMonoid(\M) \ar@<2.5pt>[r]^-{\Lbar} & \CMonoid(\N) \ar@<2.5pt>[l]^-{R}}\]
between the semi-model categories of commutative monoids in $\M$ and in $\N$.
\end{corollary}

\begin{proof}
We simply reuse the proof of Corollary \ref{cor.schwede.shipley} with $\sO$ the commutative monoid operad in $\M$, which has $\sO(n) = \tensorunitm$ for all $n \geq 0$, and $\sP$ the commutative monoid operad in $\N$. As the commutative monoid operad is not $\Sigma$-cofibrant, we need to assume $(L,R)$ is a nice Quillen equivalence, and we need to use Theorem \ref{main.theorem}.
\end{proof}

Using essentially the same proof as in the previous corollaries with the respective operads for commutative $T$-algebras and commutative $LT$-algebras, Theorem \ref{main.theorem} yields the following result.

\begin{corollary}[Commutative Algebras]
\label{cor.com.algebra}
Suppose  $L : \M \adjoint \N : R$ is a nice Quillen equivalence (Def. \ref{def:nice.qeq}) with $L$ lax symmetric monoidal, and $T$ is a commutative monoid that is cofibrant as an object in $\M$.  Then there is an induced Quillen equivalence
\[\nicexy{\calg(T) \ar@<2.5pt>[r]^-{\Lbar} & \calg(LT) \ar@<2.5pt>[l]^-{R}}\]
between the semi-model categories of commutative $T$-algebras in $\M$ and of commutative $LT$-algebras in $\N$.
\end{corollary}

This result improves on Theorem 4.19 in \cite{white-commutative}, which required $L$ to be strong symmetric monoidal.

Using instead the operads for $1$-colored non-symmetric operads in $\M$ and in $\N$, essentially the same proof yields the following special case of Theorem \ref{main.theorem.Sigma}.  A similar result for full model categories is  \cite{muro14} (1.1).  

\begin{corollary}[Non-Symmetric Operads]
\label{cor.muro}
Suppose  $L : \M \adjoint \N : R$ is a weak monoidal Quillen equivalence, that the tensor units in $\M$ and $\N$ are cofibrant, and that the generating cofibrations in $\M$ have cofibrant domains.  Then there is an induced Quillen equivalence
\[\nicexy{\omegaoperad(\M) \ar@<2.5pt>[r]^-{\Lbar} & \omegaoperad(\N) \ar@<2.5pt>[l]^-{R}}\]
between the semi-model categories of $1$-colored non-symmetric operads in $\M$ and in $\N$.
\end{corollary}

\subsection{Generalized Props}

The previous corollaries can be vastly extended to \emph{generalized props} associated to any pasting scheme in the sense of  \cite{jy2} (10.39).  We refer the reader to \cite{jy2} for detailed discussion of pasting schemes and their associated generalized props.  For any fixed set $\fC$ of colors, generalized props over a pasting scheme include: enriched $\fC$-categories ($=$ enriched categories with object set $\fC$ and object-preserving functors), $\fC$-colored operads, $\fC$-colored half-props, $\fC$-colored dioperads, $\fC$-colored prop(erad)s, $\fC$-colored wheeled operads, and $\fC$-colored wheeled prop(erad)s.  See \cite{jy2} (Chapter 11) for detailed discussion of these objects.  

When Theorem \ref{main.theorem} is restricted to the special case with $\sO$ and $\sP$ the operads for the generalized props under discussion in $\M$ and in $\N$, which are explicitly described in \cite{jy2} (14.1),  we obtain the following result.  The proof is once again basically the same as that of Corollary \ref{cor.schwede.shipley}, but uses Theorem \ref{main.theorem} because the colored operads in question are in general not $\Sigma$-cofibrant.

\begin{corollary}
\label{cor.gprop}
Suppose $L : \M \adjoint \N : R$ is a nice Quillen equivalence (Def. \ref{def:nice.qeq}) and that the tensor units in $\M$ and $\N$ are cofibrant.  Then for each pasting scheme $\fG$ in the sense of \cite{jy2} (Def. 8.2), there is an induced Quillen equivalence
\[\nicexy{\gprop(\M) \ar@<2.5pt>[r]^-{\Lbar} & \gprop(\N) \ar@<2.5pt>[l]^-{R}}\]
between the semi-model categories of $\fG$-props \cite{jy2} (10.39) in $\M$ and in $\N$.  In particular, for each color set $\fC$, there are induced Quillen equivalences between semi-model categories:
\begin{tabular}{rl}
enriched $\fC$-categories & $\catc(\M) \adjoint \catc(\N)$\\
$\fC$-colored operads & $\operadc(\M) \adjoint \operadc(\N)$\\
$\fC$-colored half-props & $\halfpropc(\M) \adjoint \halfpropc(\N)$\\
$\fC$-colored dioperads & $\dioperadc(\M) \adjoint \dioperadc(\N)$\\
$\fC$-colored properads & $\properadc(\M) \adjoint \properadc(\N)$\\
$\fC$-colored props & $\propc(\M) \adjoint \propc(\N)$\\
$\fC$-colored wheeled operads & $\woperadc(\M) \adjoint \woperadc(\N)$\\
$\fC$-colored wheeled properads & $\wproperadc(\M) \adjoint \wproperadc(\N)$\\
$\fC$-colored wheeled props & $\wpropc(\M) \adjoint \wpropc(\N)$
\end{tabular}
\end{corollary}

This result extends a result from \cite{hry15} to non-shrinkable contexts. In particular, the application to properads, colored props, and colored wheeled props, is new.

\subsection{Cyclic and Modular Operads}

Similarly, using the operads for $\fC$-colored cyclic operads or $\fC$-colored modular operads for a fixed color set $\fC$ in $\M$ and in $\N$ \cite{gk95,gk98,mss}, Theorem \ref{main.theorem} yields the following result.

\begin{corollary}
\label{cor.cyclic}
Suppose  $L : \M \adjoint \N : R$ is a nice Quillen equivalence (Def. \ref{def:nice.qeq}) and that the tensor units in $\M$ and $\N$ are cofibrant.  Then for each color set $\fC$, there are induced Quillen equivalences
\[\nicexy{\cyoperad(\M) \ar@<2.5pt>[r]^-{\Lbar} & \cyoperad(\N) \ar@<2.5pt>[l]^-{R}}
\quad
\nicexy{\modoperad(\M) \ar@<2.5pt>[r]^-{\Lbar} & \modoperad(\N) \ar@<2.5pt>[l]^-{R}}\]
between the semi-model categories of $\fC$-colored cyclic (resp., modular) operads in $\M$ and in $\N$.
\end{corollary}

\subsection{Lie algebras}

Quillen achieved a Quillen equivalence between the categories of reduced rational simplicial Lie algebras and reduced rational dg Lie algebras \cite{quillen} (p.211). We now recover this result as a special case of Theorem \ref{main.theorem}, while simultaneously generalizing Quillen's work to the setting of any field of characteristic zero.

\begin{corollary} \label{cor:Lie}
For any field $k$ of characteristic zero, the Dold-Kan equivalence $L \dashv R$ between $\calm = $ reduced simplicial $k$-modules, and $\N = $ reduced dg $k$-modules, induces a Quillen equivalence between reduced simplicial Lie $k$-algebras, and reduced dg Lie $k$-algebras.
\end{corollary}

\begin{proof}
There is an operad $\Lie_k$, in the category of $k$-vector spaces, whose algebras are Lie algebras. This operad acts in both $\calm$ and $\N$, because both are enriched in $k$-vector spaces. The categories of algebras over this operad possess transferred model structures by \cite{white-yau} (6.1.1). We now apply Theorem \ref{main.theorem}, taking $\sO$ and $\sP$ to be the operad $\Lie_k$. The map $f: \sO \to R(\sP)$ is induced by the unit of the adjunction, and $\overline{f}: L \sO \to \sP$ is an isomorphism, just as in \cite{ss03} (3.16, 4.4). That the Dold-Kan equivalence is nice is Example \ref{ex:nice}.
\end{proof}

We remark that the paper \cite{hry15} states (after Definition 5.1) that the Quillen adjunction between reduced rational simplicial Lie algebras and reduced rational dg Lie algebras is a weak monoidal Quillen pair. This is a small misprint. What the authors meant was that this Quillen equivalence could be obtained from the Dold-Kan equivalence, which is a weak monoidal Quillen equivalence. Corollary \ref{cor:Lie} verifies this claim.

\section{Applications to Left Bousfield Localization} \label{sec:bous-loc}

Left Bousfield localization is a general framework that starts with a (nice) model category $\calm$ and a set of morphisms $\C$, and produces a new model structure $L_\C(\calm)$ on the same category in which maps in $\C$ are now weak equivalences (along with all the old weak equivalences). When we say Bousfield localization we will always mean \emph{left} Bousfield localization, so cofibrations in $L_\C(\calm)$ will be the same as the cofibrations in $\calm$. The model category $\lcm$ satisfies a universal property (Theorem 3.3.20, \cite{hirschhorn}): for any left Quillen functor $F:\M \to \N$ taking the maps in $\C$ to weak equivalences, there is an induced left Quillen functor $\tilde{F}:\lcm \to \N$. 

Applications of left Bousfield localization abound: it has been used to study generalized homology theories, to create stable model structures for spectra (including equivariant and motivic spectra), for spectral sequence computations, and to give models for presentable $\infty$-categories, just to name a few. We refer the interested reader to \cite{hirschhorn} to learn more. Recently, it has become advantageous to study the interplay between Bousfield localization and operad algebra structure. A lengthy list of applications in this vein can be found in \cite{white-localization} and \cite{white-yau}.

In this section we will specialize the machinery of Theorems \ref{main.theorem} and \ref{main.theorem.Sigma} to the local setting, and prove results relating $\algolcm$ and $\algpldn$, the categories of algebras valued in local model categories. We begin with an adjunction $L:\M \adjoint \N:R$ and a class of maps $\C$ in $\M$. We define a class of maps $\D = \underline{L} \C$ in $\N$, where $\underline{L}$ is the left derived functor of the left adjoint $L$.

\subsection{Local Quillen Equivalences}

In order for operad algebras to have a well-behaved local homotopy theory, we will need $\lcm$ and $\ldn$ to be monoidal model categories. Such localizations are studied in \cite{white-thesis}, where they are called \textit{monoidal left Bousfield localizations}. In particular, the following characterization is given. We say that \textit{cofibrant objects are flat} when, for every cofibrant $X$, $X\otimes -$ preserves weak equivalences.

\begin{theorem}[Monoidal Bousfield Localization = 4.6 in \cite{white-localization}] \label{thm:PPAxiom-nontractable}
Suppose $\M$ is a cofibrantly generated monoidal model category in which cofibrant objects are flat. Then the following are equivalent:
\begin{enumerate}
\item $L_\C(\M)$ satisfies the pushout product axiom and has cofibrant objects flat.
\item Every map of the form $f \otimes \Id_K$, where $f$ is in $\C$ and $K$ is cofibrant, is a $\C$-local equivalence. 
\end{enumerate}
If the domains of the generating cofibrations are cofibrant, then it suffices to check this condition for (co)domains $K$ of the generating cofibrations.
\end{theorem}

Throughout this section, we will assume that $\M$, $\N$, $\lcm$, and $\ldn$ are cofibrantly generated monoidal model categories. We state our main result:

\begin{theorem}[Lifting Local Quillen Equivalences]
\label{main.theorem.local}
Suppose:
\begin{enumerate}
\item $L : \M \adjoint \N : R$ is a Quillen equivalence where $R$ is lax symmetric monoidal. Suppose
\begin{enumerate}
\item For all $X,Y$ cofibrant in $\M$, the comonoidal map $L(X\otimes Y)\to LX \otimes LY$ is a local weak equivalence in $\N$.
\item For some cofibrant replacement $Q \tensorunitm$ of the unit $\tensorunitm$, the composition \\
$LQ\tensorunitm \to L \tensorunitm \to \tensorunitn$ is a local weak equivalence in $\N$.
\end{enumerate} 
Recall that $\D = \underline{L}\C$.
\item $(\filledstar)$ holds in $\M$ and $\N$. The generating cofibrations in $\M$ have cofibrant domain.
\item $(\#)$ and $(\clubsuit)$ hold in $\lcm$ and $\ldn$.
\item $\N$ satisfies the local version of $(\medstar)$; i.e. for any local weak equivalence $g$ between objects of $\N^{\sigmaop_n}$ that are cofibrant in $\N$ (same as being cofibrant in $\ldn$), and for any $X$ in $\N^{\Sigma_n}$ that is cofibrant in $\N$, then $g\otimes_{\Sigma_n} X$ is a local weak equivalence.
\item $f : \sO \to R\sP$ is a map of $\fC$-colored operads in $\M$, $\sO$ an entrywise cofibrant $\fC$-colored operad in $\M$, $\sP$ an entrywise cofibrant $\fC$-colored operad in $\N$, and the entrywise adjoint $\fbar : L\sO \to \sP$ is an entrywise local weak equivalence in $\N$. 
\end{enumerate}
Then the lifted adjunction \eqref{lbar.ocomp.diagram}
\[\nicexy{\algolcm \ar@<2.5pt>[r]^-{\Lbar} & \algpldn \ar@<2.5pt>[l]^-{R}}\]
is a Quillen equivalence between the semi-model categories of $\sO$-algebras in $\lcm$ and of $\sP$-algebras in $\ldn$ (Theorem \ref{theorem623}).
\end{theorem}

We also have a streamlined version for $\Sigma$-cofibrant colored operads, that we state after the proof.

\begin{proof}
The definition of $\D$ as $\underline{L}\C$ guarantees that the adjunction $(L,R)$ descends to an adjuction $L: \lcm \adjoint \ldn : R$, by Theorem 3.3.20 in \cite{hirschhorn}. We will apply Theorem \ref{main.theorem} to this adjunction. Condition (5) is the local version of condition (2) of Theorem \ref{main.theorem}, since $\sO$ (resp., $\sP$) is entrywise cofibrant locally if and only if it is entrywise cofibrant in $\M$ (resp., $\N$). We are left to prove $L: \lcm \adjoint \ldn : R$ is a nice Quillen equivalence (Def. \ref{def:nice.qeq}). It is a weak symmetric monoidal Quillen equivalence by condition (1) of the theorem. Note that this is a weaker condition than simply assuming $(L,R)$ is a weak symmetric monoidal Quillen equivalence relative to $\M$ and $\N$. 

Next, $(\filledstar)$ only references the cofibrations, and so holds in $\M$ if and only if it holds in $\lcm$, because whenever an object $X$ is $\Sigma_n$-cofibrant in $\M$, it is $\Sigma_n$-cofibrant in $\lcm$. The same holds for $\N$. The same argument shows that the domains of the generating cofibrations in $\lcm$ are cofibrant.

We have assumed $(\#)$ and $(\medstar)$ for $\lcm$ and $\ldn$, but we note that condition (3) above is weaker than simply assuming $(\#)$ for $\M$ and $\N$, since every weak equivalence is a local weak equivalence. Similarly, assuming $(\medstar)$ is weaker than the usual method of getting a functor to preserve local weak equivalences (namely, Theorem 3.3.18 in \cite{hirschhorn}), because we do not need the functor $-\otimes_{\Sigma_n} X$ to be left Quillen.

Lastly, we have assumed $(\clubsuit)$ in $\lcm$ and $\ldn$, and this implies that \\
$L:\lcm \adjoint \ldn:R$ is a nice Quillen equivalence. Note that $\clubcof$ is the same in $\M$ and in $\lcm$.  Conditions guaranteeing $\clubtcof$ to hold in any left Bousfield localization are given in \cite{white-yau}. Examples include spaces, spectra, and chain complexes over a field of characteristic zero.
\end{proof}

We now state the version of this result for $\Sigma$-cofibrant colored operads. The proof involves applying Theorem \ref{main.theorem.Sigma} to the induced adjunction $L:\lcm \adjoint \ldn:R$, as above.

\begin{theorem}[Lifting Local Quillen Equivalences for $\Sigma$-Cofibrant Operads]
\label{main.theorem.local.Sigma}
Suppose:
\begin{enumerate}
\item $L : \M \adjoint \N : R$ is a Quillen equivalence where $R$ is lax symmetric monoidal. Suppose
\begin{enumerate}
\item For all $X,Y$ cofibrant in $\M$, the comonoidal map $L(X\otimes Y)\to LX \otimes LY$ is a local weak equivalence in $\N$.
\item For some cofibrant replacement $Q \tensorunitm$ of the unit $\tensorunitm$, the composition \\
$LQ\tensorunitm \to L \tensorunitm \to \tensorunitn$ is a local weak equivalence in $\N$.
\end{enumerate} 
Recall that $\D = \underline{L}\C$.
\item The generating cofibrations in $\M$ have cofibrant domain.
\item $f : \sO \to R\sP$ is a map of $\fC$-colored operads in $\M$, $\sO$ a $\Sigma$-cofibrant $\fC$-colored operad in $\M$, $\sP$ a $\Sigma$-cofibrant $\fC$-colored operad in $\N$, and the entrywise adjoint $\fbar : L\sO \to \sP$ is an entrywise local weak equivalence in $\N$. 
\end{enumerate}
Then the lifted adjunction \eqref{lbar.ocomp.diagram}
\[\nicexy{\algolcm \ar@<2.5pt>[r]^-{\Lbar} & \algpldn \ar@<2.5pt>[l]^-{R}}\]
is a Quillen equivalence between the semi-model categories of $\sO$-algebras in $\lcm$ and of $\sP$-algebras in $\ldn$.
\end{theorem}

\begin{remark}
Recently, the first author and Michael Batanin have investigated local operadic algebras. Theorem 3.4 of \cite{batanin-white-eilenberg}, proves that, whenever $\algolcm$ has a transferred semi-model structure, the Bousfield localization $L_{\sO(\C)} \algom$ exists and coincides with $\algolcm$, where $\sO(\C)$ denotes the free $\sO$-algebra maps on $\C$. Combining this with Theorem \ref{main.theorem.local} provides a Quillen equivalence between $L_{\sO(\C)} \algom$ and $L_{\sP(\D)} \algpn$. This can be viewed as an enhancement to Theorem 3.3.20 in \cite{hirschhorn}, as it allows for simultaneously changing the model category and the operad.
\end{remark}

\subsection{Special Cases}

Following Section \ref{sec:rect-and-change}, we provide several applications of the results above. We begin with rectification.

\begin{corollary}[Local Rectification]
\label{cor.rectification.local}
Suppose $\M$ is a cofibrantly generated monoidal model category, and $\lcm$ is a monoidal left Bousfield localization. Suppose $\M$ satisfies $(\filledstar)$ (Def. \ref{def:star}), $\clubcof$ (Def. \ref{def:club}), and every generating cofibration has a cofibrant domain. Suppose $\M$ satisfies local versions of $(\medstar)$ and $\clubtcof$ as in Theorem \ref{main.theorem.local}. Suppose $\fC$ is a set, and $f : \sO \to \sP$ is a map of entrywise cofibrant $\fC$-colored operads that is an entrywise $\C$-local weak equivalence in $\M$.  Then the induced adjunction
\[\nicexy{\algolcm \ar@<2.5pt>[r]^-{f_!} & \algplcm \ar@<2.5pt>[l]^-{f^*}}\]
is a Quillen equivalence between semi-model categories.
\end{corollary}

\begin{corollary}[Local Rectification for $\Sigma$-Cofibrant Operads]
\label{cor.rectification.local.sigma}
Suppose $\M$ is a cofibrantly generated monoidal model category, and $\lcm$ is a monoidal left Bousfield localization. Suppose that in $\M$ every generating cofibration has a cofibrant domain. Suppose $\fC$ is a set, and $f : \sO \to \sP$ is a map of $\Sigma$-cofibrant $\fC$-colored operads that is an entrywise $\C$-local weak equivalence in $\M$.  Then the induced adjunction
\[\nicexy{\algolcm \ar@<2.5pt>[r]^-{f_!} & \algplcm \ar@<2.5pt>[l]^-{f^*}}\]
is a Quillen equivalence between semi-model categories.
\end{corollary}

We turn now to modules, (commutative) monoids, (commutative) algebras, non-symmetric operads,  generalized props, cyclic operads, and modular operads. 

\begin{corollary}
Assume that the conditions of Theorem \ref{main.theorem.local} are satisfied (for $\Sigma$-cofibrant situations, Theorem \ref{main.theorem.local.Sigma} suffices).
\begin{enumerate}
\item Suppose $T$ is a (commutative) monoid that is cofibrant as an object in $\M$, and $L$ is lax (symmetric) monoidal. Then there is an induced Quillen equivalence
\[\nicexy{\Mod(T;\lcm) \ar@<2.5pt>[r]^-{\Lbar} & \Mod(LT;\ldn) \ar@<2.5pt>[l]^-{R}}\]
between the semi-model categories of left $T$-modules in $\lcm$ and of left $LT$-modules in $\ldn$.

\item Suppose $T$ is a commutative monoid that is cofibrant as an object in $\M$, and $L$ lax symmetric monoidal. Then there is an induced Quillen equivalence
\[\nicexy{\Alg(T;\lcm) \ar@<2.5pt>[r]^-{\Lbar} & \Alg(LT;\ldn) \ar@<2.5pt>[l]^-{R}}.\]

\item Suppose $T$ is a commutative monoid that is cofibrant as an object in $\M$, and $L$ lax symmetric monoidal. Then there is an induced Quillen equivalence
\[\nicexy{\calg(T;\lcm) \ar@<2.5pt>[r]^-{\Lbar} & \calg(LT;\ldn) \ar@<2.5pt>[l]^-{R}}.\]

\item Suppose the tensor units in $\M$ and $\N$ are cofibrant.  Then there is an induced Quillen equivalence
\[\nicexy{\omegaoperad(\lcm) \ar@<2.5pt>[r]^-{\Lbar} & \omegaoperad(\ldn) \ar@<2.5pt>[l]^-{R}}.\]

\item Suppose that the tensor units in $\M$ and $\N$ are cofibrant.  Then for each pasting scheme $\fG$ in the sense of \cite{jy2} (Def. 8.2), there is an induced Quillen equivalence
\[\nicexy{\gprop(\lcm) \ar@<2.5pt>[r]^-{\Lbar} & \gprop(\ldn) \ar@<2.5pt>[l]^-{R}}\]
between the semi-model categories of $\fG$-props \cite{jy2} (10.39) in $\lcm$ and in $\ldn$.

\item Suppose that the tensor units in $\M$ and $\N$ are cofibrant.  Then for each set $\fC$, there are induced Quillen equivalences
\[\nicexy{\cyoperad(\lcm) \ar@<2.5pt>[r]^-{\Lbar} & \cyoperad(\ldn) \ar@<2.5pt>[l]^-{R}}\]
and
\[\nicexy{\modoperad(\lcm) \ar@<2.5pt>[r]^-{\Lbar} & \modoperad(\ldn) \ar@<2.5pt>[l]^-{R}}.\]
\end{enumerate}

\end{corollary}

Taking $T$ to be the tensor unit in (2) and (3) implies a Quillen equivalence for local (commutative) monoids.

\section{Applications to (Commutative) $HR$-Algebras, (C)DGAs, $E_\infty$-Algebras, and Motivic Ring Spectra} \label{sec:applications}

\subsection{Dold-Kan Equivalence}

The main application of \cite{ss03} proves that the Dold-Kan equivalence lifts to categories of modules and algebras. This can be viewed as a special case of Theorem \ref{main.theorem.Sigma}, since these are categories of algebras over $\Sigma$-cofibrant operads (as explained in Section \ref{sec:rect-and-change}), since the model categories involved have generating cofibrations with cofibrant domains, and since the Dold-Kan equivalence satisfies the conditions of Theorem \ref{main.theorem.Sigma} (see Example \ref{example:ss03}). 

\subsection{(Commutative) DGAs and $HR$-Algebra Spectra}

The main theorem of \cite{shipley-hz-spectra} proves that the model categories of $HR$-algebra spectra and unbounded differential graded $R$-algebras are Quillen equivalent, where $R$ is a discrete commutative ring. Shipley lifts the chain of Quillen equivalences (with left adjoints on top)
\[\nicexy{ 
HR-Mod \quad 
\ar@<.4ex>^-{Z}[r] &
\quad  \Sp^\Sigma (\sAb) \quad 
\ar@<.4ex>^-{U}[l]
\ar@<-.4ex>_-{\phi^*N}[r] &
\quad  \Sp^\Sigma(\ch) \quad  
\ar@<-.4ex>_-{L}[l]
\ar@<.4ex>^-{D}[r] &
\quad \Ch \quad \ar@<.4ex>^-{R}[l]}\]
to the level of monoids.  Here $HR$ is the Eilenberg-Maclane spectrum, $\sAb$ is the category of simplicial $R$-modules, $\ch$ is the category of non-negatively graded chain complexes of $R$-modules, $\Ch$ is the category of unbounded chain complexes of $R$-modules with the projective model structure, and $\Sp^{\Sigma}(-)$ is a suitable category of symmetric spectra. The Quillen equivalences $(Z,U)$ and $(D,R)$ are strong symmetric monoidal, and the middle one $(L,\phi^*N)$ is weak symmetric monoidal.

Shipley's main result may be viewed as a special case of Theorem \ref{main.theorem.Sigma}, applied to each adjunction. Proposition 2.10 of \cite{shipley-hz-spectra} demonstrates that the adjunctions all satisfy the conditions of Theorem \ref{main.theorem.Sigma}.  The domains of the generating cofibrations in all settings are cofibrant.  To recover Shipley's result, in each of the three Quillen equivalences, $\sO$ and $\sP$ are both the operads whose algebras are monoids as in Corollary \ref{cor.schwede.shipley}.  A key point here is that the associative operad is $\Sigma$-cofibrant, so Theorem \ref{main.theorem.Sigma} is applicable.  Similarly, Shipley's extension to modules (Theorem 2.13 of \cite{shipley-hz-spectra}) can be viewed as a special case of Theorem \ref{main.theorem.Sigma}.

Furthermore, when $R$ has characteristic $0$ (which implies that $(\filledstar)$, $(\medstar)$  (Def. \ref{def:star}), $(\#)$ (Def. \ref{def:sharp}), and $(\clubsuit)$ (Def. \ref{def:club}) are satisfied), Theorem \ref{main.theorem} can be applied to each of Shipley's three Quillen equivalences with $\sO$ and $\sP$ the operads whose algebras are \emph{commutative} monoids as in Corollary \ref{cor.com.monoids}.  This yields, in the characteristic $0$ setting, a zig-zag of three Quillen equivalences 
\begin{equation}\label{shipley.com}
\nicexy{ C(HR-Mod)
\ar@<.4ex>^-{Z}[r] &
C\Bigl(\Sp^\Sigma (\sAb)\Bigr) 
\ar@<.4ex>^-{U}[l]
\ar@<-.4ex>_-{\phi^*N}[r] &
C\Bigl(\Sp^\Sigma(\ch)\Bigr)
\ar@<-.4ex>_-{\overline{L}}[l]
\ar@<.4ex>^-{D}[r] &
C(\Ch) \ar@<.4ex>^-{R}[l]}
\end{equation}
between the categories of commutative $HR$-algebra spectra and of commutative differential graded $R$-algebras.  This confirms a belief expressed in \cite{shipley-hz-spectra}.  As we will discuss below, a zig-zag of Quillen equivalences between the same end categories is also achieved in \cite{richter-shipley} (Corollary 8.4) using \emph{six} Quillen equivalences instead of three here.

\subsection{Commutative $HR$-Algebra Spectra, CDGAs, and $E_\infty$-Algebras}

The main theorem of \cite{richter-shipley} proves a result analogous to the above, but for commutative $HR$-algebra spectra and $E_\infty$-algebras in $\Ch$ for a discrete commutative ring $R$. The chain of Quillen equivalences produced is now:

$$ \nicexy{
{C(HR\text{-mod})} \ar@<0.5ex>[r]^Z  &
\ar@<0.5ex>[l]^U {C(\Sp^\Sigma(\sAb))}
\ar@<-0.5ex>[r]_N   &
\ar@<-0.5ex>[l]_{L_N} {C(\Sp^\Sigma(\ch))}  \ar@<0.5ex>[r]^i  &
\ar@<0.5ex>[l]^{C_0} {C(\Sp^\Sigma(\Ch))}
\ar@<-0.5ex>[d]_{\varepsilon_*}\\
 &  &
{E_\infty \Ch} \ar@<-0.5ex>[r]_-{\ev_0}
& \ar@<-0.5ex>[l]_>>>>>>>{F_0} {E_\infty(\Sp^\Sigma(\Ch))}
\ar@<-0.5ex>[u]_{\varepsilon^*}
}$$
The Quillen equivalence in the bottom row is a special case of Theorem \ref{main.theorem.Sigma}, because $E_\infty$ operads are $\Sigma$-cofibrant. The vertical Quillen equivalence is a special case of rectification. As the commutative operad is not $\Sigma$-cofibrant, we need Theorem \ref{main.theorem} in this setting. Unfortunately, we do not know if the conditions of this theorem are satisfied for symmetric spectra in $\Ch$ for general $R$.  We do, however, know that the conditions are satisfied if $R$ is replaced by a field $k$ of characteristic zero. Once this replacement is made, the vertical Quillen equivalence is a special case of Theorem \ref{main.theorem}. Of the remaining three Quillen equivalences, the outer ones are induced by strong symmetric monoidal Quillen equivalences, while the inner one $(L_N, N)$ is induced by a weak symmetric monoidal Quillen equivalence. However, in the characteristic $0$ setting, this is enough to deduce $(\#)$, $(\filledstar), (\medstar),$ and $(\clubsuit)$. Thus, in the characteristic $0$ case, all five Quillen equivalences are special cases of Theorem \ref{main.theorem}. 

Furthermore, in the characteristic $0$ setting, there is a rectification Quillen equivalence between $E_\infty$-algebras in differential graded modules $E_\infty\Ch$ and commutative differential graded algebras $C(\Ch)$.  In this case, the above zig-zag is prolonged to a zig-zag of six Quillen equivalences between commutative $Hk$-algebra spectra and commutative differential graded $k$-algebras, which is Corollary 8.4 in \cite{richter-shipley}.  However, as discussed in the previous section \eqref{shipley.com}, using the main results of this paper, we actually have a zig-zag with the same end categories involving only three Quillen equivalences, which are Shipley's original adjunctions.

\subsection{Motivic Applications}

The final application of \cite{dmitri-motivic} constructs a strictly commutative ring spectrum for Deligne cohomology. This is done by pushing a commutative differential graded algebra through a chain of Quillen equivalences terminating with strictly commutative motivic ring spectra. However, since both the starting category and the ending category admit rectification, one could instead view the CDGA as an $E_\infty$-algebra in chain complexes, then use Theorem \ref{main.theorem.Sigma} to push this object through a chain of lifted Quillen equivalences, and then use rectification in the positive stable model structure on motivic symmetric spectra (Theorem 3.10, \cite{hornbostel}) to strictify the resulting $E_\infty$-algebra into a strictly commutative ring spectrum.

\subsection{Dual Dold-Kan Equivalence} \label{sec:dualDold}

The dual Dold-Kan equivalence is studied in \cite{cortinas}, where the authors study the classical Quillen pair

\[
K: \Ch^{\geq 0} \leftrightarrows \Ab^{\Delta}: N
\]

between non-negatively graded cochain complexes $\Ch^{\geq 0}$ and cosimplicial abelian groups, where $N$ is the normalized complex functor (also known as the Moore complex), and $K$ is its left adjoint. They replace $K \dashv N$ with the following Quillen pair, which we will refer to as the \textit{monoidal dual Dold-Kan equivalence}:

\[
Q: \Ch^{\geq 0} \leftrightarrows \Ab^{\Fin}: P
\]

where $Q$ is strong symmetric monoidal \cite{cortinas} (5.2, 7.5), and $\Fin$ has the same objects as $\Delta$, but morphisms are any set morphisms. By restricting the $\Fin$-structure to a cosimplicial structure, they obtain a left Quillen equivalence $\tilde{Q}$ that is derived-equivalent to $K$. They explicitly lift $\tilde{Q}$ to a Quillen equivalence between differential graded rings and cosimplicial rings \cite{cortinas} (9.6). This result may be recovered and generalized as a special case of our Theorem \ref{main.theorem}, as we now explain.

\begin{corollary} \label{cor:dual-Dold}
Let $k$ be any commutative ring. Then the monoidal dual Dold-Kan equivalence is a weak monoidal Quillen equivalence and lifts to a Quillen equivalence between differential graded $k$-algebras and cosimplicial $k$-algebras. 
\end{corollary}

\begin{proof}
First, the construction of the adjoint pair $Q\dashv P$ can be carried out in $\Ch(k)^{\geq 0}$ just as well as in $\Ch(\mathbb{Z})^{\geq 0}$, as can the proof that $Q$ is strong symmetric monoidal. Since the monoidal units of $\Ch(k)^{\geq 0}$ and $(\Vect_k)^{\Fin}$ are cofibrant, $Q\dashv P$ is a weak symmetric monoidal Quillen equivalence \ref{def:weak.symmetric.monoidal}, and hence the result is a consequence of Corollary \ref{cor.schwede.shipley}.
\end{proof}

We may also expand the work of \cite{cortinas} to other colored operads.

\begin{corollary}
Let $k$ be a field of characteristic zero. Then the monoidal dual Dold-Kan equivalence lifts to Quillen equivalences 
\begin{enumerate}
\item between commutative differential graded $k$-algebras and cosimplicial commutative $k$-algebras; 
\item between differential graded Lie $k$-algebras and cosimplicial Lie $k$-algebras.
\end{enumerate}
\end{corollary}

\begin{proof}
Because $k$ has characteristic zero, the monoidal dual Dold-Kan equivalence is a Nice Quillen equivalence (Example \ref{ex:nice}). The result now follows in precisely the same way as Corollaries \ref{cor.com.monoids} and \ref{cor:Lie}.
\end{proof}


\end{document}